\documentclass[11pt]{amsart} 
\usepackage{enumerate} 
\usepackage{amsmath,amsfonts, amscd,  eucal, latexsym, amssymb,   rotating,
booktabs}
 \usepackage[utf8]{inputenc} 
 \usepackage[T1]{fontenc}
 \usepackage{mathrsfs}

\DeclareFontFamily{OT1}{pzc}{}
\DeclareFontShape{OT1}{pzc}{m}{it}{<-> s * [1.200] pzcmi7t}{}
\DeclareMathAlphabet{\mathpzc}{OT1}{pzc}{m}{it}

\usepackage{tikz-cd}
\usetikzlibrary{patterns}

\usepackage[margin=1.25in]{geometry}
\input xy
 \xyoption{all}

 \usepackage[colorlinks=true, pdfstartview=FitV,
linkcolor=cyan, citecolor=red, urlcolor=blue]{hyperref}

\newif\ifPDF
\ifx\pdfoutput\undefined %Step to check pdf on older tetex version
        \ifx\pdfversion\undefined
                \PDFfalse
        \fi
\else %Newer version of tetex defines pdfoutput in both cases, so need
        \ifnum\pdfoutput > 0
                \PDFtrue
        \else
                \PDFfalse
        \fi
\fi

\makeatletter
\newcommand{\xRightarrow}[2][]{\ext@arrow 0359\Rightarrowfill@{#1}{#2}}
\makeatother

\theoremstyle{plain}
				
\newtheorem{theorem}{Theorem}[section]				
\newtheorem{proposition}[theorem]{Proposition}		
\newtheorem{corollary}[theorem]{Corollary}
\newtheorem{lemma}[theorem]{Lemma}
\newtheorem{claim}[theorem]{Claim}

\theoremstyle{definition}

\newcommand{\Mbold}{{\bf M}}

\newcommand{\CBbb}{\mathbb C}

\newcommand{\RBbb}{\mathbb R}

\newcommand{\ZBbb}{\mathbb Z}

\newcommand{\Ccal}{\mathcal C}

\newcommand{\Fcal}{\mathcal F}

\newcommand{\Hcal}{\mathcal H}

\newcommand{\Ncal}{\mathcal N}

\newcommand{\Escr}{\mathscr E}

\newcommand{\Hscr}{\mathscr H}

\newcommand{\Sp}{\mathsf{Sp}}

\DeclareMathOperator{\real}{Re}

\DeclareMathOperator{\tr}{tr}

\DeclareMathOperator{\Det}{Det}

\DeclareMathOperator{\Log}{Log}

\DeclareMathOperator{\DN}{\mathsf{DN}}

\DeclareMathOperator{\I}{\mathscr{I}}

\newcommand{\dbar}{\bar\partial}

\begin{document}

\title[The determinant of the Dirichlet-to-Neumann operator]
{A Meyer-Vietoris formula for the determinant of the Dirichlet-to-Neumann operator 
on Riemann surfaces}

\author[Wentworth]{Richard A. Wentworth}
\address{Department of Mathematics,
   University of Maryland,
   College Park, MD 20742, USA}
\email{raw@umd.edu}
\thanks{The author gratefully acknowledges support from NSF grants
DMS-1906403 and DMS-2204346.}
%\thanks{The authors also acknowledge support from NSF grants DMS-1107452, -1107263, -1107367 ``RNMS: GEometric structures And Representation varieties'' (the GEAR Network).
%}
%Primary: 58D15; Secondary: 14D20, 32G13
%
\dedicatory{Dedicated to Peter Li  on the occasion of his 70th
birthday}

\begin{abstract}
This paper presents a Meyer-Vietoris type gluing formula for a conformal invariant of a 
    Riemannian surface with boundary that is defined by the 
    determinant of the Dirichlet-to-Neumann operator.  
    The formula is used to bound the asymptotics of the invariant under degeneration. 
    It is shown that the associated height function on the moduli space of
    hyperbolic surfaces with geodesic boundary is proper only in genus
    zero.
    Properness implies  a compactness theorem for Steklov isospectral metrics
   in the case of genus zero.
    The formula also provides
    asymptotics for the determinant of 
    the Laplacian with Dirichlet or Neumann  boundary conditions. 
    For the proof, we derive an extension of Kirchhoff's weighted matrix tree
    theorem for graph Laplacians with an external potential.
\end{abstract}
\maketitle

\allowdisplaybreaks

\thispagestyle{empty}

\section{Statement of results}
Let $(M,g)$ be a connected compact oriented Riemannian surface with
nonempty boundary of total length $\ell(\partial M)$.
The \emph{Dirichlet-to-Neumann  operator} $\DN(M)$ is a self-adjoint pseudo-differential
operator of order $1$ in $L^2(\partial M)$, which contains a great deal of
information about $M$. In fact, $\DN(M)$ essentially determines the 
conformal class of the metric $g$, and hence the Riemann
surface structure of $M$ (see \cite{LassasUhlmann:01}, and also
\cite{HenkinMichel:07}). 
The spectrum of $\DN(M)$ consists of the \emph{Steklov eigenvalues}, 
which have been studied by many authors. For a survey, see \cite{GPPS:14},
\cite{GP:17}, or \cite{SteklovSurvey}, and the references therein.

Because $\DN(M)$ is a self-adjoint elliptic operator of positive order,
the formal product of the  Steklov eigenvalues 
can be defined via the   
  zeta-regularized determinant $\Det^\ast(\DN(M))$, according to the
Ray-Singer procedure \cite{RaySinger:73} (the asterisk indicates
that the zero eigenvalue has been omitted).  
Determinants of elliptic operators often  package global information through heat
invariants in a way that reflects the geometry more transparently than the
individual eigenvalues themselves. 
The determinant of the Dirichlet-to-Neumann operator for planar domains 
was first considered
by Edward and Wu \cite{EdwardWu:91}, who showed that for a simply 
connected bounded domain 
$\Omega\subset\CBbb$ with smooth boundary,
$\Det^\ast(\DN(\Omega))=\ell(\partial\Omega)$.
More generally,  
in \cite[Theorem 1.1]{GuillarmouGuillope:07}, Guillarmou and Guillop\'e proved that for any
compact Riemannian surface $M$  with boundary, 
\begin{equation} \label{eqn:I}
\I(M):=\frac{\Det^\ast(\DN(M))}{\ell(\partial M)}
\end{equation}
is a conformal invariant of the metric on $M$. 
Eq.\ \eqref{eqn:I}  therefore defines  a positive, real-valued  function of the 
  underlying biholomorphic equivalence class of the  Riemann surface structure.
  We shall refer to $\I(M)$ as the \emph{DN-invariant} of a Riemann surface
  $M$ with boundary. 

  What kind of information  does the DN-invariant provide about $M$?
To give some idea,   in 
 this paper we prove a Meyer-Vietoris type expression for $\I(M)$ similar to the
 gluing formulas obtained  in \cite{BFK:92, Forman:87, Wentworth:12}.
Using this we show that $\I(M)$ detects the existence of long thin
separating
cylinders in $M$ (or in terms of hyperbolic geometry, short geodesics), at
least when there are relatively few  ``isolated'' components.  
This in turn leads to a compactness theorem for families of  genus zero hyperbolic
surfaces with geodesic boundary and
the same Steklov eigenvalue spectrum. 
We make the observation that the  invariant $\I(M)$ 
 may be expressed as a \emph{ratio} of determinants of
Laplace operators on $M$ with Neumann and Dirichlet conditions (see
Proposition \ref{prop:I-ratio} below), whereas by doubling 
the surface we have an expression for  the \emph{product} of these determinants.
Therefore, combining the asymptotic properties
arising from the Meyer-Vietoris theorem in this paper with the results of
Wolpert for closed hyperbolic surfaces  (see \cite{Wolpert:87}), we obtain
bounds on the asymptotic
behavior of the determinant of the Laplacian with Dirichlet and Neumann
boundary conditions.

Let us briefly explain the gluing formula. 
Suppose $\Gamma\subset M$ is a disjoint collection of oriented closed 
smoothly embedded curves $\gamma:S^1\hookrightarrow M$
not meeting $\partial M$. 
 Let $M_\Gamma$ denote the (possibly disconnected) surface with boundary obtained from $M\setminus \Gamma$ by adjoining two boundary components 
for every component of $\Gamma$.  By $\I(M_\Gamma)$ we mean the 
product of the DN-invariants over the connected components of $M_\Gamma$.
Let $\Ncal(M,\Gamma)$ denote the Neumann jump operator in
$L^2(\Gamma)$ (cf.\ \cite{BFK:92}).
  This is simply the pairwise sum of the  DN operators
associated to $M_\Gamma$ for each component in $\Gamma$. We emphasize that 
in defining $\Ncal(M,\Gamma)$, 
\emph{Dirichlet} conditions are imposed  on $\partial M$.
Let
$\Ncal_A(M,\Gamma)$ the jump operator on $M_\Gamma$ defined in
\cite{Wentworth:12} for
the trivial bundle with trivial framing and
\emph{Alvarez} boundary conditions on $\partial M$.
More precisely, this is defined for a general surface with boundary as follows.
For complex valued functions $\varphi$, let $\varphi'$, $\varphi''$ denote the real and  imaginary parts, respectively.  
Also, let $\imath: \partial M\hookrightarrow M$ be the inclusion and $\ast$ 
the Hodge operator of the metric induced on the boundary.  Then  we define
\begin{equation} \label{eqn:dn}
\DN_A(M)(f,g):=((\ast\imath^\ast(\dbar\varphi))', \varphi'\circ\imath)
\ ,
\end{equation}
where $(f,g)$ are real functions on $\partial M$,  and $\varphi$ is harmonic on $M$ satisfying 
$$(\varphi''\circ\imath , (\ast\imath^\ast(\dbar \varphi))'')=(f,g)\ .$$
In the case of $M_\Gamma$, 
the jump operator $\Ncal_A(M,\Gamma)$ in $L^2(\Gamma)$ is then the pairwise sum of  $\DN_A(M_\Gamma)$ operators as in the case of $\Ncal(M,\Gamma)$.
We emphasize that the harmonic extension is such that $\varphi''$ satisfies Dirichlet boundary conditions on $\partial M$, whereas $\varphi'$ satisfies Neumann conditions.
The operator $\Ncal_A(M,\Gamma)$ plays the role of the Neumann jump operator for mixed Dirichlet-Robin type boundary conditions.
A key difference, however, is that $\Ncal_A(M,\Gamma)$ 
is a pseudo-differential operator of order zero.
Its determinant is defined as in
\cite{FriedlanderGuillemin:08} with the aide of an
auxiliary  pseudo-differential operator $Q$ of order $1$.
 For simplicity, in this paper  we fix a particular $Q$ with
$\zeta_Q(0)=0$ once and for all (cf.\ \eqref{eqn:Q}).  See \cite{Wentworth:12} for more details.

Finally,  associated to $\Gamma$ is a simple graph 
$G_\Gamma$ whose vertices $V(G_\Gamma)$
are the connected components $\{M_i\}$ of
$M\setminus \Gamma$ and whose edges $E(G_\Gamma)$ are the intersections $\partial M_i\cap
\partial M_j$ for distinct $i$ and $j$, 
which correspond to  a union of possibly more than one component of
$\Gamma$.
The metric $g$ on $M$ 
gives a \emph{weight function}  on $G_\Gamma$; 
namely, a map $\omega_{g} : E(G_\Gamma)\to \RBbb^+$ obtained by 
setting $\omega_{g}(ij)$ to be
the sum  total  of the lengths of the components of $\Gamma$ in $\partial M_i\cap
\partial M_j$.
Let $\Delta_{(G_\Gamma,\omega_{g})}$  be the associated weighted graph
Laplacian.\footnote{Graph Laplacians will appear 
throughout the paper. Notation and the relevant results   can be found in Section
\ref{sec:graphs}, which is independent of the other sections of this paper.}
We note in passing that this type of graph Laplacian has previously appeared in the study of
small eigenvalues of Laplace operators on
hyperbolic surfaces (see \cite{SWY:80,DPRS:87, Burger:90}).
 We  now can state the main result.
\begin{theorem} \label{thm:main}
For any compact connected oriented Riemannian surface $M$ with $\partial M\neq \emptyset$, and $\Gamma\subset M$ as above, the following holds:
$$
    \I(M)=\I(M_\Gamma)\left({\det}^\ast\Delta_{(G_\Gamma,\omega_{g})}\right)\frac{(2\cdot 
    \Det^\ast_Q \Ncal_A(M,\Gamma))}{(\Det \Ncal(M,\Gamma))^2} \ . 
$$
\end{theorem}
\noindent
In Section \ref{sec:examples}, we
illustrate Theorem \ref{thm:main} with an explicit computation for the disk
and annulus, and we use 
the result to obtain asymptotic formulas in the case of multiply connected
planar domains (see Theorem \ref{thm:planar}). 

As discussed in \cite{Wentworth:91b,Wentworth:08,Wentworth:12}, 
gluing formulas of the type in Theorem \ref{thm:main} above are 
convenient for computing the asymptotic behavior of determinants. 
In Section \ref{sec:asymptotics}, we use this to prove  bounds on $\I(M)$
for degenerating hyperbolic surfaces. This is contained in
Theorems
\ref{thm:asymptotics} and \ref{thm:higher-genus} below.
As a consequence, following  \cite{OPS:89} we define
the 
 \emph{height function}
on the moduli space of Riemann surfaces with boundary by
\begin{equation*}\label{eqn:height}
    \Hscr(M):=-\log\I(M) \ .
\end{equation*}
Then we have the following 

\begin{theorem} \label{thm:proper}
Fix  integers $g\geq 0$ and $n\geq 1$, where 
 $n\geq 3$ if $g=0$.
    Let $\Mbold(g;b_1,\ldots, b_n)$ denote the moduli space of hyperbolic
    surfaces of genus $g$ and geodesic boundaries of lengths $b_1,\ldots, b_n$.
    Then $\Hscr(M)$ is a proper function on $\Mbold(g;b_1,\ldots, b_n)$
    if and only if $g=0$. 
\end{theorem}
\noindent
The above is analogous to the result of Osgood-Phillips-Sarnak and 
Khuri on flat surfaces for the
height function associated to the Laplacian with Dirichlet boundary
conditions \cite{OPS:89,Khuri:91}.

It is conjectured that nonequivalent surfaces 
with the same Steklov spectrum are rare (see
\cite{Jollivet:14}). Unfortunately, the height $\Hscr(M)$ does not seem to provide
much information on this question in higher genus. In genus zero, however, 
as in \cite{OPS:89} we can draw the following consequence.
\begin{corollary}\label{cor:steklov}
    Fix a surface $M$ of genus zero with 
    at least three
    boundary components. 
Let $\Fcal$ be a family of hyperbolic metrics on $M$ with geodesic
    boundary 
    that are mutually Steklov isospectral. 
    Then $\Fcal$ is precompact in the $C^\infty$
    topology. 
\end{corollary}
\noindent
A compactness theorem for Steklov isospectral simply connected planar domains has been proven
in \cite{Edward:93} and \cite{Jollivet:18}.

  We also obtain asymptotic results for  
determinants of Laplace operators  on $M$. 
For a hyperbolic surface $M$ with nonempty geodesic boundary,
let $[\Det \Delta_D]_M$ and $[\Det^\ast\Delta_N]_M$ denote the zeta-regularized
determinants of the Laplace operators $\Delta_D$ and $\Delta_N$ on $M$ with  
 Dirichlet and Neumann boundary    conditions on $\partial M$,
respectively. We set $\{\kappa_i\}$ to be the collection of all eigenvalues for
\emph{both}
Dirichlet and Neumann problems on $M$ that satisfy $0<\kappa_i<1/4$. 
As a consequence
of Theorem \ref{thm:asymptotics} and work of Wolpert on the asymptotic behavior of the
Selberg zeta function \cite{Wolpert:87},
in Corollary \ref{cor:asymptotics} below we give bounds on the asymptotic
behavior of $[\Det \Delta_D]_M$ and $[\Det^\ast\Delta_N]_M$. 
In particular, we have the following. 

\begin{corollary} \label{cor:dirichlet-height}
Let $\Hscr_{D}(M):=-\log
    [\Det\Delta_{D}]_M$
denote the height function for the  determinant of the Laplacian with
    Dirichlet boundary conditions on $\partial M$.
Then $\Hscr_{D}(M)$ is a proper function on 
$\Mbold(g;b_1,\ldots, b_n)$.
\end{corollary}

We point out that  this result has already been obtained by Young-Heon Kim.
The estimate found in this paper,
$$
\Hscr_D(M)\geq \sum_{\gamma\in \Gamma} \left(
\frac{\pi^2}{3\ell(\gamma)}+\frac{3}{2}\log\ell(\gamma)\right)+\frac{1}{2}\sum_i
\log(1/\kappa_i)-\log C\ ,
$$
is on the one hand sharper than that in \cite[Thm.\
3.3]{Kim:08}, and in particular it incorporates the small eigenvalues (as suggested
should be possible in
\cite[Rem.\ 3.3]{Kim:08}). In the case of genus zero, we also obtain 
a more precise statement, and an upper bound (see Corollary
\ref{cor:asymptotics}).
On the other
hand, in Corollary \ref{cor:dirichlet-height} the boundary lengths are fixed whereas Kim's result does not
assume this. It may be possible to obtain the
additional terms in \cite[Thm.\ 3.3]{Kim:08} that account for varying
boundary lengths
using the methods here, but we have not pursued this.

Finally, in order to prove 
Theorem \ref{thm:proper} we found it necessary to derive a general
formula for the determinant of a weighted graph Laplacian with a positive
diagonal potential (see Theorem \ref{thm:kirchhoff}). The result is  an
extension of Kirchhoff's weighted matrix tree theorem \cite{Kirchhoff}, and thus  it may be of independent
interest.

\section{The gluing formula} \label{sec:gluing}

\subsection{Preliminary remarks on determinants}
Here we briefly review some facts about determinants of operators. 
For a strictly positive, self-adjoint, elliptic pseudodifferential operator $A$ of
positive order acting  on a Hilbert space $V$ of functions (or sections of
a bundle) on a 
compact manifold $M$, possibly with boundary and elliptic boundary
conditions, the complex power $A^{-s}$ is trace-class for $\real s\gg 0$
(see \cite{Seeley:67,Seeley:69}).
The trace $\zeta_A(s)=\tr A^{-s}$ has a meromorphic continuation to the
plane and is regular at $s=0$. The \emph{zeta-regularized determinant of
$A$} is defined as:
\begin{equation} 
    \label{eqn:determinant}
    \Det A := \exp\left( -\zeta_A'(0)\right)
    \ .
\end{equation}
This is extended to operators with a kernel by defining $\Det^\ast
A$ via the restriction of the trace  to $(\ker A)^\perp$. 

Let $A(\varepsilon)$ be a differential family of such (invertible) operators with
$$A=A(0)\ ,\ B=\frac{d}{d\varepsilon}A(\varepsilon)\bigr|_{\varepsilon=0}\
, $$
and suppose $A^{-1}B$ is trace-class.
Then
\begin{equation} \label{eqn:trace-class}
\frac{d}{d\varepsilon}\log \Det
A(\varepsilon)\biggr|_{\varepsilon=0}=\tr(A^{-1}B)
\ .
\end{equation}

Determinants of more general operators are defined in
\cite{KontsevichVishik:93}. For example,
if $A$ has order zero, then a determinant may be defined as follows (see
\cite{FriedlanderGuillemin:08} for details). 
Define
$$
\Log A :=\frac{i}{2\pi}\int_C dz(\log z)(z-A)^{-1}
\ , 
$$
where $\log$ is the principal branch and the contour $C$ contains the
spectrum of $A$. 
Pick a positive self-adjoint
pseudodifferential operator $Q$ on $M$ of order 1. 
We then define:
\begin{equation} \label{eqn:det-def}
    \log\Det_Q A := f.p.\, \tr(Q^{-s}\Log A)\biggr|_{s=0}
    \ , 
\end{equation}
where ``f.p.'' denotes the finite part of the meromorphic extension.
If $A$ is not positive, by convention we set (see \cite[p.\ 479]{Wentworth:12}):
$$
\log\Det_Q A :=\frac{1}{2}\log\Det_Q(A^2)\ .
$$
The formula for variations \eqref{eqn:trace-class} continues to hold with
this definition of determinant. 

Finally, we need the following result whose proof is straightforward. Suppose that $A$ is of
positive order and acts on $V$.
Let $\pi$ be the orthogonal projection  operator to a finite dimensional subspace $V_0$,
and $V_1=\ker(1-\pi)$.
With respect to the splitting $V=V_0\oplus V_1$, write
$$
A=\left(\begin{matrix} A_0 & B^\dagger \\ B& A_1\end{matrix}\right)\ .
$$
\begin{lemma} \label{lem:subspace-det}
    Assume $A_1$ is invertible. Then
    the operator $A_1^{-s}$ on $V_1$ is trace class for $\real s\gg 0$. 
    The zeta function $\zeta_{A_1}(s)$ has a meromorphic continuation that
    is  regular at $s=0$.
    If $\Det A_1:=\exp(-\zeta_{A_1}'(0))$, then 
    \begin{equation} \label{eqn:det-split} 
    \Det A=\det(A_0-B^\dagger A_1^{-1}B)\Det(A_1)\ .
    \end{equation} 
    Eq.\ \eqref{eqn:det-split} holds for zero-th order operators $A$ as
    well, replacing $\Det$ by $\Det_Q$,  where we assume that $Q$ preserves
    the splitting $V_0\oplus V_1$. 
\end{lemma}
%
%Now we turn to asymptotic behavior.  We will need the following
%
%\begin{lemma} \label{lem:product-det}
%With respect to the decomposition $L^2(\Gamma)=\oplus_{i=1}^n L^2(\gamma_i)$, let
%$$P=\left(\begin{matrix} P_1 &&0 \\ &\ddots& \\ 0&& P_n \end{matrix}\right)
%$$
%be a diagonal pseudo-differential operator of positive order on $L^2(\Gamma)$.  Let $R$ be a trace-class operator such that $P_t=P+tR$ is uniformly invertible for $0\leq t\leq 1$.  Finally, let
%$$\lambda=\left(\begin{matrix} \lambda_1 &&0 \\ &\ddots& \\ 0&& \lambda_n \end{matrix}\right)\ , \ \lambda_i>0
%$$
%be a scalar matrix.  Then
%$$
%\Det\left( \lambda(P+R)\right)=\bigl(\prod_{i=1}^n \lambda_i^{\zeta_{P_i}(0)}\Det P_i\bigr)(1+O(\Tr R))
%$$
%\end{lemma}
%
%\begin{proof}
%Since $\dot P_t=R$ and $P_t$ is invertible, $P_t^{-1}\dot P_t$ is trace-class. Hence,
%$$
%\frac{d}{dt}\log \Det\left( \lambda P_t\right)=\Tr (P_t^{-1} R)
%$$
%and so integrating $t$ from $0$ to $1$ yields
%$$
%\log\Det\left( \lambda(P+R)\right)=\log\Det\left( \lambda P\right)+O(\Tr R)
%$$
%On the other hand, since $\lambda P$ is diagonal,
%$$
%\log\Det\left( \lambda P\right)=\sum_{i=1}^n \log \Det(\lambda_i P_i)=\sum_{i=1}^n\left\{(\log \lambda_i)\zeta_{P_i}(0)+ \log\Det P_i\right\}
%$$
%This completes the proof.
%\end{proof}
%

\subsection{Proof of Theorem \ref{thm:main}} \label{sec:main}
Let us return to the context of the Introduction, where $M$ is a compact
Riemannian surface with nonempty boundary.
We begin by giving a different expression for the invariant
$\I(M)$.

\begin{proposition} \label{prop:I-ratio}
Let $A(M)$ denote the area of $M$ and $\kappa_g$ the geodesic
curvature of $\partial M$.  Then
\begin{equation} \label{eqn:weisberger}
\I(M)=\frac{\Det^\ast \Delta_N}{A(M)\Det
\Delta_D}\exp\left(-\frac{1}{2\pi}\int_{\partial M}\kappa_g
ds\right)\ .
\end{equation}
\end{proposition}

\begin{proof}
By the Polyakov-Alvarez formula \cite{Alvarez:83}, the right hand side of
    \eqref{eqn:weisberger} is conformally invariant (see \cite[eqs.\ (4)
    and (5)]{Weisberger:87}).  Hence, it
suffices to prove \eqref{eqn:weisberger} in the case of geodesic
boundary.  Let $\widehat M$ be the double of $M$ along
$\partial M$.  Then by decomposing the spectrum with respect to the
    isometric involution on $\widehat M$, we have
\begin{equation} \label{eqn:double}
\left[\Det^\ast \Delta\right]_{\widehat M}=\left[ \Det
\Delta_D\Det^\ast\Delta_N\right]_{M}
\ .
\end{equation}
On the other hand, by \cite[Theorem B*]{BFK:92}, cutting $\widehat M$ along
$\partial M$ gives
    \begin{equation} \label{eqn:bfk}
[\Det^\ast \Delta]_{\widehat M}=
 \left[\Det \Delta_D\right]^2_{M} \frac{A(\widehat M)}{\ell(\partial
M)}\Det^\ast(2\DN(M))
\ .
\end{equation}
Now $A(\widehat M)=2A(M)$,  and since
$\zeta_{\DN(M)}(0)=-1$ (cf.\ \cite{GuillarmouGuillope:07, Wentworth:08}),
    we have
$$\Det^\ast(2\DN(M))=\frac{1}{2}\Det^\ast(\DN(M))\ .$$ The result
    now follows from \eqref{eqn:double}, \eqref{eqn:bfk}, and the definition
    \eqref{eqn:I}.
\end{proof}

Let $\Gamma\subset M$ be as in the Introduction, and recall
that $\Gamma$ is assumed to be oriented. Hence, the boundary
components of $M_\Gamma$ are signed, depending upon whether the
orientations induced from the outward normals agree with those of $\Gamma$.
This allows us to define a map $\delta_\Gamma$ from functions on
$\partial M_\Gamma$ to functions on $\Gamma$ by taking the difference of
the values on the two sheets of the double cover $\partial M_\Gamma\to
\Gamma$.  With this understood, we have the following.

\begin{lemma} \label{lem:psi}
    Let $M_i$, $i=1,\ldots, p$, denote the connected components of
    $M_\Gamma$. 
Let $\{\Psi_j\}_{j=1}^p$ be any basis of  locally constant functions
on $M_\Gamma$.
Then
$$
\frac{\det^\ast(\delta\Psi_i,\delta\Psi_j)_\Gamma}{\det(\Psi_i,
\Psi_j)_{M_\Gamma}}
    = \frac{{\det}^\ast \Delta_{(G_\Gamma,\omega_{g})}
}{\prod_{i=0}^p
A(M_i)}\ .
    $$
\end{lemma}

\begin{proof}
The statement is independent of the choice of basis, so choose
$
\Psi_i
$
to be the characteristic function on $M_i$. 
Clearly,
$
    \det(\Psi_i, \Psi_j)_{M_\Gamma}=\prod_{i=1}^p A(M_i)
$.
We also have
$$
(\delta\Psi_i,\delta\Psi_j)_\Gamma
=\begin{cases} \ell(\partial M_i\cap \Gamma)^\ast & i=j\ , \\
-\ell(\partial M_i\cap \partial M_j) & i\neq j \ .
\end{cases}
$$
The $\ast$ means we omit components of $\Gamma$ that bound only $M_i$.
Now the result follows 
    by the definition of $(G_\Gamma, \omega_g)$ and  the graph Laplacian
    from the Introduction (see Section \ref{sec:graphs} for more
    details).
\end{proof}

\begin{proof}[Proof of Theorem \ref{thm:main}]
Use a small modification of \cite[Theorem B*]{BFK:92} to incorporate the boundary $\partial M$, and write
$$
\left[\Det\Delta_D\right]_{M}=\prod_{i=0}^p
\left[\Det\Delta_D\right]_{M_i} \Det \Ncal(M,\Gamma) \ .
$$
On the other hand, by a similar modification of \cite[Theorem
    3.3]{Wentworth:12} we have
$$
    \frac{
\left[\Det\Delta_D\Det^\ast\Delta_N\right]_{M}
    }{2A(M)}
=\prod_{i=0}^p\left[\Det\Delta_D\Det^\ast\Delta_N\right]_{M_i}
\frac{\det^\ast(\delta\Psi_i,\delta\Psi_j)_\Gamma}{\det(\Psi_i,
\Psi_j)_{M_\Gamma}}
\Det^\ast_Q \Ncal_A(M,\Gamma) \ .
$$
    The result now follows from Proposition \ref{prop:I-ratio} and Lemma \ref{lem:psi}. 
    Notice that all the factors involving geodesic curvature along $\Gamma$ cancel pairwise.
\end{proof}

\section{Examples} \label{sec:examples}

\subsection{Disks and annuli}
By \cite{EdwardWu:91}  and \cite{GuillarmouGuillope:07}, or alternatively
\cite{Weisberger:87}
and eq.\ \eqref{eqn:weisberger}, or indeed by a direct calculation,  it follows that $\I(M)=1$
for the disk, and $\I(M)=2\pi/\log\rho$ for the annulus with modulus
$(\log\rho)/2\pi$,  $1<\rho<\infty$. 
If $M$ is a euclidean
disk of radius $R$ and $\Gamma$ the circle centered at the origin of radius $r$, $0<r<R$, then Theorem \ref{thm:main} states in this case that
\begin{equation} \label{eqn:disk}
1=\frac{4\pi}{\log\rho}(2\pi r)\frac{\Det_Q^\ast \Ncal_A}{(\Det \Ncal)^2}
\ ,
\end{equation}
where $\rho=R/r$ and $\Ncal$, $\Ncal_A$ 
denote the operators for the pair $(M,\Gamma)$, and we have used that 
$\det^\ast\Delta_{(G_\Gamma,\omega_g)}=2\pi r$ in this case. Let us 
verify \eqref{eqn:disk} directly.  
The operators $\Ncal$, $\Ncal_A$ have eigenvalues $1/r\log\rho$, and 
$1/2r\log\rho$, corresponding to the constant functions on $\Gamma$. 
It therefore suffices to show that
\begin{equation} \label{eqn:temp}
(\Det' \Ncal)^2=(2\pi r)^2\Det_Q' \Ncal_A
\ ,
\end{equation}
where the prime indicates the determinant of the operator restricted to the space $L^2_0(\Gamma)\subset L^2(\Gamma)$ orthogonal to the constants.  For the operator $Q$ we may choose
\begin{equation} \label{eqn:Q}
Q\bigl( \sum_{n\in \ZBbb} f_ne^{in\theta}\bigr)=f_0+ \sum_{n\neq 0} |n|f_ne^{in\theta}
\ ,
\end{equation}
so that $\zeta_Q(0)=0$ as in the Introduction.
Since $\zeta_{\Ncal}(0)=-1$, we must  prove
\begin{equation} \label{eqn:temp-bis}
    \log\Det' \widehat\Ncal=\frac{1}{2}\log \Det_Q' \Ncal_A \ ,
\end{equation}
where $\widehat\Ncal=(2\pi r)\Ncal$.
Set $\tau=\rho^{-1}$.

Now by a direct calculation one finds (recall \eqref{eqn:dn}):
\begin{align*}
\Ncal\bigl(\sum_{n\neq 0} f_ne^{in\theta}\bigr)&=\sum_{n\neq
    0}\frac{|n|a_n}{r} f_ne^{in\theta}\ , \\
\Ncal_A\biggl(\sum_{n\neq 0} {f_n\choose g_n}e^{in\theta}\biggr)&=\sum_{n\neq 0} a_n
\left(\begin{matrix} 0& -i\sigma(n)\\ i\sigma(n) & -\frac{2r}{|n|}\end{matrix}\right)
{f_n\choose g_n}e^{in\theta} \ ,
\end{align*}
where $ a_n=\frac{2}{1-\tau^{2|n|}}$ and $\sigma(n)$ is the sign of $n$.  
For fixed $R$, regard $\widehat\Ncal$ and $\Ncal_A$ as operators depending upon
$\tau$. 
Then as $\tau\to 0$, 
$$
\zeta_{\widehat\Ncal(0)}(s)=2(4\pi)^{-s}\zeta(s)\ ,\ (\Ncal_A(0))^2=4\cdot\text{id}
\ ,
$$
where $\zeta(s)$ is the Riemann zeta function.  Below we use that 
$\zeta(0)=-1/2$ and $\zeta'(0)=-\frac{1}{2}\log(2\pi)$.
Hence, on the one hand, 
\begin{align*}
    \log\Det_Q'\Ncal_A(0)&=\frac{1}{2}\log\Det_Q'(\Ncal_A(0))^2
    =\frac{1}{2}\, f.p.\tr\left(Q^{-s}\log(4)\cdot\text{id}\right)\bigr|_{s=0}
    \\
    &=4\log(2)\, f.p.(\zeta(s))\bigr|_{s=0}
    =-2\log(2)
    \ .
\end{align*}
On the other hand, 
$$-\zeta_{\widehat\Ncal(0)}'(0)=2(\log(4\pi)\zeta(0)-2\zeta'(0))=-\log(2)\
.$$
Thus, \eqref{eqn:temp-bis} is satisfied in this case. 

Next, by direct calculation,
$$
\frac{d}{d\tau}\log\Det_Q'\Ncal_A=
f.p.\tr\left(Q^{-s}\widehat\Ncal^{-1}\frac{d}{d\tau}\widehat\Ncal\right)\biggr|_{s=0}
\ .
$$
But since the derivative of $a_n$ with respect to $\tau$ vanishes rapidly
with $n$, the operator $\widehat\Ncal^{-1}(d\widehat\Ncal/d\tau)$ is trace-class, and
so 
$$
    \frac{d}{d\tau}\log\Det'\widehat\Ncal=
\tr\left(\widehat\Ncal^{-1}\frac{d}{d\tau}\widehat\Ncal\right)
    =f.p.\tr\left(Q^{-s}\widehat\Ncal^{-1}\frac{d}{d\tau}\widehat\Ncal\right)\biggr|_{s=0}
    =
\frac{d}{d\tau}\log\Det_Q'\Ncal_A
\ .
$$
Thus, \eqref{eqn:temp-bis} is proven.

\subsection{Multiply connected planar domains}
\label{sec:planar}
Next, consider a planar domain 
$M(\boldsymbol{\varepsilon})$, $\boldsymbol{\varepsilon}=(\varepsilon_1, \ldots, \varepsilon_n)$,
obtained by removing disks of radius $\varepsilon_i>0$ (with fixed centers at
$a_1,\ldots, a_n$) from a  disk of fixed radius, which  without loss of 
generality we may choose to be the unit disk $D$.
The following generalizes the case of the annulus in the previous section.

%\begin{theorem} \label{thm:planar}
%    There is a constant $C$ depending upon $\{a_1,\ldots, a_n\}$ and
%    $\varepsilon_0$ such that for all $\boldsymbol{\varepsilon}$
%    where each $\varepsilon_i\leq\varepsilon_0/2$,
%$$
%    C^{-1}\leq 
%    \I(M(\boldsymbol{\varepsilon}))\prod_{i=1}^n\log(1/\varepsilon_i)
%\leq C
%$$
%\end{theorem}

\begin{theorem} \label{thm:planar}
$\displaystyle
    \lim_{\boldsymbol{\varepsilon}\to 0} 
    \I(M(\boldsymbol{\varepsilon}))\prod_{i=1}^n\log(1/\varepsilon_i)
    =(2\pi)^n$.
\end{theorem}

\begin{proof}
Fix an $\varepsilon_0>0$, $\varepsilon_0<\frac{1}{2}\min\{ |a_i-a_j|
\mid i\neq j\}$, and   
such that $D_{\varepsilon_0}(a_i)\Subset D$.
Then set
$\boldsymbol{\varepsilon}_0=(\varepsilon_0,\ldots,\varepsilon_0)$. 
    We suppose $\boldsymbol{\varepsilon}< \boldsymbol{\varepsilon}_0$, and
    let $A_{\varepsilon_i}$ denote the annulus $\varepsilon_i\leq
    |z-a_i|\leq \varepsilon_0$.
    Recall that
    $\I(A_{\varepsilon_i})=2\pi/\log(\varepsilon_0/\varepsilon_i)$. 
    Then taking $\Gamma$ in Theorem
    \ref{thm:main} to be the collection of curves $\partial
    D_{\varepsilon_0}(a_i)$, we have:
    \begin{equation} \label{eqn:disk-factorize}
\I(M(\boldsymbol{\varepsilon}))
=
\I(M(\boldsymbol{\varepsilon}_0))
        ({\det}^\ast\Delta_{(G_\Gamma,\omega_g)})\frac{(2\cdot\Det_Q^\ast
        \Ncal_A(M(\boldsymbol{\varepsilon}),\Gamma))}{(\Det\Ncal(M(\boldsymbol{\varepsilon}),\Gamma))^2}
        \times \prod_{i=1}^n \frac{2\pi}{\log(\varepsilon_0/\varepsilon_i)}
        \ .
    \end{equation}

%    \begin{figure} 
%% \resizebox{140pt}{120pt}{
%\begin{tikzpicture}
%    \hspace{-5.8cm}
%\draw [pattern=north east lines, pattern color=gray!30]
%(0,0) circle [radius=1.3]; 
%    \draw [fill=white!] (-.6,-.3) circle [radius=.2];
%    \draw [fill=white!] (.6,-.3) circle [radius=.2];
%    \draw [fill=white!] (0,.5) circle [radius=.2];
%   \draw [-stealth] (1.5,0) -- (2.5,0);
%\hspace{4cm}
%\draw [pattern=north east lines, pattern color=gray!30]
%(0,0) circle [radius=1.3]; 
%\draw [fill=blue!30] (-.6,-.3) circle [radius=.3];
%    \draw [fill=white!] (-.6,-.3) circle [radius=.15];
%\draw [fill=blue!30] (.6,-.3) circle [radius=.3];
%    \draw [fill=white!] (.6,-.3) circle [radius=.15];
%\draw [fill=blue!30] (0,.5) circle [radius=.3];
%    \draw [fill=white!] (0,.5) circle [radius=.15];
%   \draw [-stealth] (1.5,0) -- (2.5,0);
%\hspace{4cm}
%\draw [pattern=north east lines, pattern color=gray!30]
%(0,0) circle [radius=1.3]; 
%\draw [fill=blue!10] (-.6,-.3) circle [radius=.3];
%\draw [fill=blue!10] (.6,-.3) circle [radius=.3];
%\draw [fill=blue!10] (0,.5) circle [radius=.3];
%   \draw [-stealth] (1.5,0) -- (2.5,0);
%\hspace{4cm}
%\draw [pattern=north east lines, pattern color=gray!30]
%(0,0) circle [radius=1.3]; 
%\end{tikzpicture}
%%}
%        \caption{}
%    \end{figure}

    Since the first two  factors  on the right hand side of
    \eqref{eqn:disk-factorize} are fixed independent of the
    $\varepsilon_i$, 
    it suffices to analyze the determinants of the Neumann jump
    operators. 
    But by direct calculation, the Dirichlet-to-Neumann operator on an
    annulus of  modulus $\rho$ (with Dirichlet conditions on the inner
    boundary) is 
    equal to the Dirichlet-to-Neumann operator on the disk of radius
    $\varepsilon_0$ up to a trace-class operator whose norm $\to 0$ as
    $\varepsilon_i\to 0$. To be precise, on the $n$-th Fourier mode,
    $\DN(A_{\varepsilon})$ acts as:
    $$
    \DN(A_{\varepsilon})(g_ne^{in\theta})=\frac{|n|}{\varepsilon_0}
    \left(1+\frac{2\rho^{-2|n|}}{1-\rho^{-2|n|}}\right)
    g_ne^{in\theta}\ .
    $$
    It follows that as $\varepsilon_i\to 0$, 
    $
\Det\Ncal(M(\boldsymbol{\varepsilon}),\Gamma)
    $ tends to the determinant of the corresponding operator on the domain with
    $A_{\varepsilon_i}$ replaced by the disk  $D_{\varepsilon_0}(a_i)$. 
The same holds for $\Ncal_A(M(\boldsymbol{\varepsilon}),\Gamma)$. 
Applying the gluing formula once again to the limit, we obtain the result.
\end{proof}

\section{Asymptotics of $\I(M)$} \label{sec:asymptotics}

\subsection{Genus zero}
As we have seen in Section \ref{sec:planar}, general bounds on the invariant $\I(M)$
are obtained from Theorem \ref{thm:main} by judicious choices of $\Gamma$. 
This can be done on a case-by-case basis; to get overall uniform estimates is
combinatorially complicated. For simplicity, here we treat only the genus zero
case in complete generality. In higher genus, we find rough estimates
that suffice for the application to the Laplacian with Dirichlet
conditions. 
Recall that 
by a ``short geodesic'' on a hyperbolic surface, 
we mean a simple closed geodesic of length less than the absolute constant
$c_0$ appearing in the collar lemma (cf.\ \cite{Keen:74}).

\begin{theorem} \label{thm:asymptotics}
Fix  positive numbers
$b_1,\cdots, b_n$,  $n\geq 3$, and $\delta>0$.
    Then there is a constant $C\geq 1$ 
     depending only on $c_0$, $(b_1,\cdots, b_n)$, and $\delta$, 
such that the following holds.
For any hyperbolic surface $M$ of genus zero with
geodesic boundary components of lengths
$b_1,\cdots, b_n$, let $\Gamma$ be the
    collection of short geodesics on $M$ and
    $\Delta_{(G_\Gamma,\omega_{g})}$
    the graph Laplacian from the Introduction.
    Then
    \begin{equation} \label{eqn:asymptotics}
    C^{-1}
        \frac{\prod_{\gamma\in
        \Gamma}\ell(\gamma)}{\det(\Delta_{(G_\Gamma,\omega_{g})}+D)}
    \leq
    \I(M)\leq C
        \frac{\prod_{\gamma\in \Gamma}\ell(\gamma)
        }{\det(\Delta_{(G_\Gamma,\omega_{g})}+D)}
        \ ,
    \end{equation} 
    where 
         $D$ is the diagonal matrix with entries $\delta$
    for the vertices
    corresponding to the  components of $M_\Gamma$ that intersect $\partial M$, and zeros
    elsewhere. 
\end{theorem}
%This theorem implies, in particular, that $\I(M)$ is bounded among all
%hyperbolic surfaces with geodesic boundary components of fixed lengths. 
%A natural question is: For which such  surfaces is $\I(M)$ maximized?

\begin{proof}
The proof proceeds in several steps. 
Recall that we denote the connected components of $M\setminus \Gamma$ by
$M_i$. 

\medskip
\noindent {\bf  Step 1}.
First,
about each $\gamma\in \Gamma$  we choose a  conformal
    annulus $A_\gamma$ of modulus $(\log \rho(\gamma))/2\pi$,
    $1<\rho(\gamma)<+\infty$.
    We choose a  local coordinate 
$1/\rho(\gamma)\leq |z|\leq 1$ for $A_\gamma$. By a conformal change 
we adjust  the hyperbolic  metric   to be
euclidean in a neighborhood of the boundary $\partial A_\gamma$.
    Introduce  the notation
    \begin{equation} \label{eqn:lambda}
    \lambda(\gamma):=1/\log\rho(\gamma)\ .
    \end{equation}
    The significance of
    $\lambda(\gamma)$ is that it is the value of the $\DN$-operator for
    $A_\gamma$  at the
    boundary $|z|=1$ for the characteristic function of this boundary.
    The value the $\DN$-operator at the other boundary is
    $-\lambda(\gamma)$. 
    By the collar lemma and  Wolpert's estimate (cf.\ \cite{Wolpert:90}),
    we may choose $A_\gamma$ such that  the length $\ell(\gamma)$ (in the hyperbolic
    metric) of the geodesic in $(M,g)$
    homotopic to $\gamma$ satisfies $\ell(\gamma)\sim 2\pi^2/\log\rho(\gamma)$. 
    Thus, we have a comparison between $\lambda(\gamma)$ and
    $\ell(\gamma)$ that is uniform as $\ell(\gamma)\to 0$.

Let 
$$\widetilde\Gamma := \bigcup_{\gamma\in \Gamma}\partial A_\gamma\ ,$$
and set $\widetilde M_i$ to be the connected component of 
    $$\mathring{M}_{\widetilde\Gamma}:=M\setminus
    \bigcup_{\gamma\in \Gamma }A_\gamma$$ that intersects $M_i$. 
    We let $\overline M_{\widetilde\Gamma}$ denote the (possibly
    disconnected) Riemann  surface obtained by ``capping off''
    $\mathring{M}_{\widetilde\Gamma}$, i.e.\ replacing each 
    annulus $A_\gamma$ by a pair of disks.  
An important point is the bounded geometry of $\widetilde M_i$ as $M$ varies in the
moduli space of hyperbolic metrics with fixed set $\Gamma$ of short geodesics. See
\cite{Wolpert:87,Wolpert:90}. 
    In the following, we shall call $\widetilde M_i$ an \emph{isolated
    component} if $\partial\widetilde M_i\cap\partial M=\emptyset$.

\medskip
\noindent {\bf  Step 2}.
Let $\Ncal=\Ncal(M,\widetilde\Gamma)$ and
$\Ncal_A=\Ncal_A(M,\widetilde\Gamma)$ be the Neumann jump operators 
acting on $L^2(\widetilde\Gamma)$. 
For each $i$, let $\chi_i\in L^2(\widetilde\Gamma)$
denote the function defined by boundary values of the characteristic function of the component
    $\widetilde M_i$, i.e.\ $\chi_i$ is locally constant  on $\widetilde
    \Gamma$,  and is equal to $1$
    on $\partial \widetilde M_i$ and $0$ otherwise. 
Let us write an orthogonal decomposition
$$
L^2(\widetilde\Gamma)=V_0\oplus V_1\ , 
$$
where $V_0$ is the span of 
all $\chi_i$ for isolated components $\widetilde M_i$. 
By a similar analysis to the one carried out in \cite{Wentworth:91b,
Wentworth:08, Wentworth:12}, for example, on the orthogonal complement
$V_1$,  as the
    $\ell(\gamma)\to 0$ 
the operators $\Ncal$ and $\Ncal_A$ converge in trace-class
to the corresponding operators $\overline \Ncal$ and $\overline\Ncal_A$ on the capped off components 
of
$\overline M_{\widetilde \Gamma}$ (see also the proof of Lemma
\ref{lem:key} below). 

\medskip
\noindent {\bf Step 3}.
We must analyze the small eigenvalues of $\Ncal$ and
$\Ncal_A$, which occur from the restriction to $V_0$. 
For $\Ncal_A$, this refers to the $\varphi''$ component. The analysis is
    therefore identical for both $\Ncal$ and $\Ncal_A$, and so henceforth
    we deal only with $\Ncal$.  
For the isolated components $\widetilde M_i$,
the harmonic function on $\widetilde M_i$ with the boundary values of $\chi_i$
is $\chi_i$ itself, and so the Dirichlet-to-Neumann operator for
$\widetilde M_i$  annihilates this. 
    On the other hand, if $\partial \widetilde M_i\cap \partial
    M\neq\emptyset$, then since Dirichlet conditions are imposed on
    $\partial M$, the $\DN$ operator is nonzero on the boundary
    values of $\chi_i$.  
    For each collar $A_\gamma$, $\DN$ is rotationally
    symmetric, and so preserves constants. 
In terms of the splitting $V_0\oplus V_1$, we may write:
\begin{equation} \label{eqn:matrix} 
\Ncal=
\left(
    \begin{matrix}
        A & B^\dagger\\ B&\Ncal_0
    \end{matrix}
\right)\ .
\end{equation}
Now $\Ncal_0$ is uniformly invertible as  the lengths $\ell(\gamma)\to 0$. 
By Lemma \ref{lem:subspace-det}, we have
\begin{equation} \label{eqn:modified-det}
\Det\Ncal=\det(A-B^\dagger\Ncal_0^{-1}B)\Det(\Ncal_0)\ . 
\end{equation}
It follows that
$\Det\Ncal$ is estimated by 
$\det(A-B^\dagger \Ncal_0^{-1}B)$. 
Recall the weighted graph $(G_\Gamma, \omega_{g})$ from the Introduction. 
The key result is the following

\begin{proposition} \label{prop:A}
    Fix $\delta>0$.
    There is a constant $C\geq 1$  depending only on $c_0$, $(b_1,\ldots,
    b_n)$, and $\delta$,  such that
    $$
    C^{-1}\det(\Delta_{(G_\Gamma,\omega_{g})}+D)
    \leq \det(A-B^\dagger\Ncal_0^{-1}B)\leq 
    C\det(\Delta_{(G_\Gamma,\omega_{g})}+D)
    \ .
    $$
\end{proposition}
\noindent
We postpone the proof of this proposition to the next section.

\medskip
\noindent {\bf Step 5}.
Assuming Proposition \ref{prop:A}, we complete the proof of Theorem
\ref{thm:asymptotics}.
By definition:
$$
\I(M_{\widetilde\Gamma})=
\I(\mathring{M}_{\widetilde\Gamma})
\prod_{\gamma\in \Gamma} \I(A_\gamma)
\ .
$$
Now $\I(A_\gamma)=2\pi/\log\rho(\gamma)\simeq \ell(\gamma)/\pi$, and
because $\mathring{M}_{\widetilde\Gamma}$ has bounded
geometry over the moduli space, 
$\I(\mathring{M}_{\widetilde\Gamma})$ is bounded from above and below away
from zero. Hence, there is an estimate (above and below):
$$
\I(M_{\widetilde\Gamma})
\simeq C\prod_{\gamma\in \Gamma}\I(A_\gamma)
\simeq C\prod_{\gamma\in \Gamma}\ell(\gamma)
\ .
$$
 Apply Theorem \ref{thm:main} to $\widetilde \Gamma$. 
The lengths of the elements of $\widetilde \Gamma$
are bounded away from
zero, so the factor ${\det}^\ast\Delta_{(G_{\widetilde\Gamma}, \omega_g)}$ in Theorem
\ref{thm:main} remains bounded
above and below away from zero. 
From the discussion in Step 2 above, eq.\ \eqref{eqn:modified-det},  and 
Proposition \ref{prop:A},
we have
$$
\frac{\Det_Q^\ast\Ncal_A(M,\widetilde\Gamma)}{(\Det\Ncal(M,\widetilde\Gamma))^2}\sim
\frac{1}{\Det\Ncal(M,\widetilde\Gamma)}
\ ,
$$
and
$$
C^{-1}\det(\Delta_{(G_\Gamma,\omega_M)}+D)
\leq \Det\Ncal(M,\widetilde\Gamma)
\leq C\det(\Delta_{(G_\Gamma,\omega_M)}+D)
\ ,
$$
for a constant $C$.
Putting this all together completes the proof.
\end{proof}

\subsection{Proof of Proposition \ref{prop:A}}
Here, we relate the action of $\Ncal$ on the locally  constant functions 
in $L^2(\widetilde \Gamma)$ associated to the isolated components of
$\mathring{M}_{\widetilde\Gamma}$ to the
graph Laplacian $\Delta_{(G_\Gamma,\omega_M)}$.
This relationship is contained in the matrix $A$ appearing in
\eqref{eqn:matrix}.  
Now the key point  is that in terms of the graph Laplacian, the  form of
the modification to the matrix $A$ in Proposition \ref{prop:A},
 $B^\dagger\Ncal_0^{-1}B$, corresponds  
 to adding edges 
with weights that are at least quadratic in the weights of the graph $G_\Gamma$. 
Using the results in Section \ref{sec:graphs}, we then  argue that for
small weights  such a modification gives only a small perturbation of determinants. 

To spell this out precisely, let us introduce some convenient notation. 
Enumerate the components $\widetilde M_i$ of
$\mathring{M}_{\widetilde\Gamma}$, 
and let $L_{ij}$ denote the adjacency matrix of $G_\Gamma$.
By the assumption of genus zero, if $L_{ij}\neq 0$, then there is a unique
element $\gamma_{ij}\in \Gamma$ bounding $M_i$ and $M_j$ (by definition:
$\gamma_{ij}=\gamma_{ji}$). 
Associated to $\gamma_{ij}$ are two elements of $\widetilde \Gamma$, one
bounding $\widetilde M_i$ and the other $\widetilde M_j$. 
Let $\chi_{i,\gamma_{ij}}$ and  $\chi_{j,\gamma_{ij}}$ denote the
characteristic functions of these two  components of $\widetilde \Gamma$.
\begin{figure}
    \hspace{-.5cm}
    \begin{tikzpicture}
        \draw (-3,1) to [out=-25, in=205] (3,1);
        \draw (-3,-1) to [out=25, in=155] (3,-1);
        \draw [thick] (0,.265) to [out=-75, in=75] (0,-.265);
        \draw (-2,.6) to [out=-75, in=75] (-2,-.6);
        \draw (2,.6) to [out=-75, in=75] (2,-.6);
        \draw [dotted] (-2,.6) to [out=-100, in=100] (-2,-.6);
        \draw [dotted] (2,.6) to [out=-100, in=100] (2,-.6);
        \node at (-3.5, 0) {$\widetilde M_i$};
        \node at (3.5, 0) {$\widetilde M_j$};
        \node at (0, .5) {$\gamma_{ij}$};
        \node at (-1.75, -.9) {$\chi_{i,\gamma_{ij}}$};
        \node at (2, -1) {$\chi_{j,\gamma_{ij}}$};
    \end{tikzpicture}
    \caption{}
\end{figure}
We define a weight function on $G_\Gamma$ by
\begin{equation}  \label{eqn:omega-lambda}
    \omega_M(ij) :=\begin{cases} \lambda(\gamma_{ij}) & L_{ij}\neq 0\\
    0 & L_{ij}=0\end{cases}\ ,
\end{equation} 
where $\lambda(\gamma)$ is defined in \eqref{eqn:lambda} above.
By the discussion in Step 1 of the previous section, there is a constant
$\kappa\geq 1$, uniform as the $\ell(\gamma)\to 0$, such that
for all $i,j$,
\begin{equation} \label{eqn:comp}
    \kappa^{-1} \omega_g(ij)\leq \omega_M(ij)\leq
    \kappa\omega_g(ij)\ .
\end{equation}
Let $G_0\subset G_\Gamma$ be the (possibly disconnected) subgraph obtained
by deleting nonisolated vertices and their edges. 
We identify the vertices of $G_0$ with the basis elements $\chi_i$ of
$V_0$. Then we  
\begin{claim} \label{claim:A}
    The matrix $A$ in \eqref{eqn:matrix} is the restriction to $G_0$ of the graph Laplacian 
    $\Delta_{(G_\Gamma,\omega_M)}$.
\end{claim}

\begin{proof}
    Let $\widetilde M_i$ and $\widetilde M_j$ be  isolated components. 
    The claim amounts to the statement that
    \begin{align}
        \begin{split}  \label{eqn:basic}
        \langle \Ncal(\chi_i), \chi_j\rangle &= -L_{ij}\lambda(\gamma_{ij})
    \ ,    \\
            \langle \Ncal(\chi_i), \chi_i\rangle &= \sum_{k\text{
                isolated}} L_{ik}\lambda(\gamma_{ik})
        + \sum_{k\text{
                not isolated}} L_{ik}\lambda(\gamma_{ik})
\ . 
        \end{split}
    \end{align}
This follows by direct calculation of the harmonic extensions of the
    locally constant functions $\chi_i$ (cf.\ Step 1 of the previous
    section). 
\end{proof}

We next consider the other entries of the decomposition \eqref{eqn:matrix}.
 As in \eqref{eqn:basic}, we have
\begin{equation} \label{eqn:B}
    B\chi_i = \sum_jL_{ij}\lambda(\gamma_{ij})\left\{
    (\chi_{i,\gamma_{ij}})^\perp-
    (\chi_{j,\gamma_{ij}})^\perp\right\} \ .
\end{equation}
Here, $\perp$ indicates the orthogonal projection to $V_1$ in the decomposition
$V_0\oplus V_1$. Note that the sum is over \emph{all} components
$\widetilde M_j$, not just isolated ones.  
Also, by definition $\Ncal_0\chi_i=0$. Let us define 
    \begin{equation} \label{eqn:P}
        P_{ij}:=
    \sum_{k,k'} L_{ik}L_{jk'}\lambda(\gamma_{ik})\lambda(\gamma_{jk'})
    \left\langle  \Ncal_0^{-1}\chi_{i,\gamma_{ik}}-\Ncal_0^{-1}\chi_{k,\gamma_{ik}},
    \chi_{j,\gamma_{jk'}} - \chi_{k',\gamma_{jk'}}\right\rangle
\end{equation}

Since  the summand in  \eqref{eqn:P} is skew-symmetric in $j$ and $k'$, we
have $\sum_j P_{ij}=0$, and therefore
\begin{equation} \label{eqn:miracle}
P_{ii}=-\sum_{j\neq i} P_{ij}\ .
\end{equation}
Let $\widehat G_\Gamma$ denote the complete graph on the vertices of
$G_\Gamma$; similarly $\widehat G_0\subset\widehat G_\Gamma$ denotes the complete graph on $G_0$.
Define weights for $\widehat G_\Gamma$ (possibly zero or nonpositive) by:
\begin{equation} \label{eqn:def-weights}
\widehat\omega_M(ij)=
\omega_M(ij) 
+ P_{ij}\ ,\ i\neq j\ .
\end{equation}
If we set (see \eqref{eqn:laplacian})
$$
\widehat\mu_M(i)=\mu_M(i)-P_{ii}\ ,
$$
then it follows from \eqref{eqn:miracle} that
$$\widehat\mu_M(i)
=\sum_{j\neq i} \widehat\omega_M(ij)
\ .
$$
Moreover, if $\widetilde M_i$ and $\widetilde M_j$ are both isolated, then
from \eqref{eqn:B} we have
$
P_{ij}=  \langle \Ncal_0^{-1}B\chi_i, B\chi_j\rangle 
$.
Combining this with Claim \ref{claim:A}, we therefore have the following:

\begin{claim} \label{claim:B}
    The matrix $A-B^\dagger\Ncal_0^{-1}B$  is the restriction to $\widehat G_0$ of the graph Laplacian 
    $\Delta_{(\widehat G_\Gamma, \widehat\omega_M)}$.
\end{claim}

By the discussion in Section \ref{sec:graphs} below (cf.\
\eqref{eqn:induction}), we conclude: 
\begin{lemma} \label{lem:reduction}
Fix $\delta>0$, and let $D$ denote the diagonal matrix with entry $\delta$
for all nonisolated components, and zeros elsewhere.
    Then
$    \det(A-B^\dagger\Ncal_0^{-1}B)$ is the coefficient of $\delta^k$ in
    $\det(\Delta_{(\widehat G_\Gamma, \widehat\omega_M)}+D)$,
    where $k$ is the number of nonisolated components. 
\end{lemma}

Given distinct vertices $v_i, v_j\in V(G_\Gamma)$, then since 
$G_\Gamma$ is a tree there is a unique geodesic $g_{ij}$ in $G_\Gamma$ from
$v_i$ to $v_j$. Moreover, there is a 1-1 correspondence
$\Gamma\leftrightarrow E(G_\Gamma)$.  For $\gamma\in \Gamma$, we shall say $\gamma\in g_{ij}$ if
the edge associated to $\gamma$ lies on $g_{ij}$. 

\begin{lemma} \label{lem:key}
    Fix $\varepsilon_0$, $0<\varepsilon_0<1$. 
    There are constants $\kappa\geq 1$ and $C>0$ such that
    if $\lambda(\gamma)$ is sufficiently small with respect to  $\varepsilon_0$ for all $\gamma\in \Gamma$,
     then:
    \begin{equation} \label{eqn:key1}
       \kappa^{-1} \omega_M(ij)\leq \widehat \omega_{M}(ij)\leq
        \kappa  \omega(ij)\ ,\  L_{ij}\neq 0\ ;
    \end{equation}
    \begin{equation} \label{eqn:key2}
        |\widehat \omega_{M}(ij)|\leq C\prod_{\gamma\in g_{ij}}
        \lambda(\gamma)\ , \ L_{ij}= 0\ .
    \end{equation}
\end{lemma}

\begin{proof}
By explicit computation, we may write
$
\Ncal_0=\overline\Ncal_0+R
$,
where $R$ is diagonal with respect to the orthogonal decomposition
    $V_1\cap \oplus_{\gamma\in \Gamma} L^2(\partial A_\gamma)$, 
    and the component pieces $R_\gamma\to 0$ in trace class as
    $\lambda(\gamma)\to 0$. 
If $f_n$ denotes the $n$-th Fourier mode of a function $f$ on one boundary
component of $\partial A_\gamma$, $n\neq 0$,  then the $n$-th Fourier mode of
    $R_\gamma(f)$ on this component is
    $$
    \frac{|n|\rho^{-2|n|}}{1-\rho^{-2|n|}}f_n\ ,
    $$
    and on the other component of $\partial A_\gamma$ it is
    $$
    -\frac{2|n|\rho^{-|n|}}{1-\rho^{-2|n|}}f_{-n}\ .
    $$
    In particular, given $\varepsilon_0$ then for sufficiently small $\lambda(\gamma)$ the  norm of $R$ is  bounded by 
      $\varepsilon_0\lambda(\gamma)$. 
    Since $\overline\Ncal_0^{-1}$ is uniformly
    bounded, we have
$
    \Ncal_0=\overline\Ncal_0(1+\overline\Ncal_0^{-1}R)
$,
where $\overline\Ncal_0^{-1}R$ has small norm in trace class bounded on
    each component by $\ell(\gamma)$. Let $f$ be supported on  
    $\partial \widetilde M_i$ and $g$ on $\partial \widetilde M_j$, and fix
    $p\geq 1$. We
    claim there is a constant $C'>0$, independent of $i,j, p$ and the
    $\lambda(\gamma)$, such that
    $$
    |\langle(\overline\Ncal_0^{-1}R)^p f, g\rangle|\leq C'\varepsilon_0^p\Vert f\Vert\Vert g\Vert
    \prod_{\gamma\in g_{ij}}\lambda(\gamma)\ .  
    $$
    This follows easily by induction on $p$.
    We now apply this estimate, and use the expression
    $$
    \Ncal_0^{-1}=\overline\Ncal_0^{-1}+\sum_{p=1}^\infty (-1)^p
    (\overline\Ncal_0^{-1}R)^p \overline\Ncal_0^{-1} 
    $$
    in the definition \eqref{eqn:P} of $P_{ij}$. 
For example,  for $i\neq j$, one of the terms is  
    \begin{align}
        \sum_{k,k'}L_{ik}L_{jk'}\lambda(\gamma_{ik})\lambda(\gamma_{jk'})\langle \Ncal_0^{-1}\chi_{k,\gamma_{ik}}, \chi_{k',\gamma_{jk}}\rangle
        &= \sum_{k}L_{ik}L_{jk}\lambda(\gamma_{ik})\lambda(\gamma_{jk})
    \langle \overline\Ncal_0^{-1}\chi_{k,\gamma_{ik}}, \chi_{k,\gamma_{jk}}\rangle
        \notag \\
        &\quad +
        \sum_{k,k'}L_{ik}L_{jk'}\lambda(\gamma_{ik})\lambda(\gamma_{jk'})O\bigl(\prod_{\gamma\in
        g_{kk'}} \lambda(\gamma)\bigr)\ ,
        \label{eqn:mess}
    \end{align}
    where we have used the fact that  $\overline
    \Ncal_0^{-1}(\chi_{k,\gamma_{ik}})$ 
    is supported on $\partial\widetilde M_k$ for any $i$. 
    Notice that  the  second term
    on the right hand side of \eqref{eqn:mess} contains $\lambda(\gamma)$
    for every $\gamma\in g_{ij}$, and therefore satisfies both (i) and (ii). 
    The first term vanishes if $L_{ij}\neq 0$, and so automatically
    satisfies (i). If $L_{ij}=0$, it
    contributes $\lambda(\gamma_{ik})\lambda(\gamma_{jk})$ for each vertex
    $k$ ``subadjacent'' to $i$ and $j$ (see Figure 2). In particular, this term satisfies 
    (ii).
    The other terms in  \eqref{eqn:P} are treated similarly. 
\end{proof}

\begin{figure}
    \hspace{-.8cm}
    \begin{tikzpicture}
        \draw (0,0) -- (2,1.5);
        \draw (2,1.5) -- (4,0);
        \draw[dotted] (0,0) -- (4,0);
        \node at (0,0) {$\bullet$};
        \node at (2,1.5) {$\bullet$};
        \node at (4,0) {$\bullet$};
        \node at (-.2,-.2) {$i$};
        \node at (2,1.8) {$k$};
        \node at (4.2,-.2) {$j$};
        \node at (.5,1) {$\lambda(\gamma_{ik})$};
        \node at (3.5,1) {$\lambda(\gamma_{jk})$};
        \node at (2,-.3) {$\lambda(\gamma_{ik})\lambda(\gamma_{jk})$};
       % \node at (2, .3) {added edge};
    \end{tikzpicture}
    \caption{}
\end{figure}

Finally,  we complete the proof of Proposition \ref{prop:A}. Fix $\delta>0$, and
suppose $M_\Gamma$ has
$k$ nonisolated components. 
By the result of  Lemma \ref{lem:reduction} and the expansion
\eqref{eqn:egregium},
if the weights $\lambda(\gamma)$ are sufficiently small compared to
$\delta$ for all $\gamma\in \Gamma$, then  the determinants are dominated by
the $\delta^k$ coefficients. 
It  therefore suffices to relate 
    $\det(\Delta_{(\widehat G_\Gamma, \widehat\omega_M)}+D)$ to
    $\det(\Delta_{(G_\Gamma, \omega_g)}+D)$.
    From  \eqref{eqn:key2}  and  Corollary \ref{cor:comparison},
    there is a constant $C\geq 1$ such that 
    $$
C^{-1}\det(\Delta_{(G_\Gamma, \widetilde \omega_M)}+D)
\leq
\det(\Delta_{(\widehat G_\Gamma,\widehat \omega_M)}+D)
\leq
C\det(\Delta_{(G_\Gamma, \widetilde \omega_M)}+D)
\ ,
$$
where $\widetilde\omega_M$ is the restriction of the weight function
$\widehat\omega_M$ on $\widehat G_\Gamma$ to $G_\Gamma$. 
By \eqref{eqn:key1} and Corollary 
\ref{cor:increasing}, there is constant $C_1\geq 1$ such that
$$
C_1^{-1}\det(\Delta_{(G_\Gamma, \omega_M)}+D)
\leq
\det(\Delta_{(G_\Gamma, \widetilde \omega_M)}+D)
\leq
C_1\det(\Delta_{(G_\Gamma, \omega_M)}+D)
\ .
$$
Finally, using \eqref{eqn:comp} and the same comparison 
between $\omega_M$ and $\omega_g$ gives upper and
lower bounds on
$\det(\Delta_{(G_\Gamma, \omega_g)}+D)$.
Combining these statements completes the proof of
Proposition \ref{prop:A}.

\subsection{Higher genus}
In higher genus, the graph $G_\Gamma$ will not be a tree in general. This
leads to a more complicated perturbation of the graph Laplacian. 
Nevertheless, 
it is clear from the proof of Theorem \ref{thm:asymptotics}
that there is a uniform upper bound on $\Det\Ncal(M,\widetilde\Gamma)$. 
Indeed, the discussion concerned the low eigenvalues, whereas as the
operator on the orthogonal complement in the previous section converges in trace-class. 
As a consequence, directly from Theorem \ref{thm:main}, we have the following
\begin{theorem} \label{thm:higher-genus}
Fix  positive numbers $g\geq 1$ and 
$b_1,\cdots, b_n$,  $n\geq 1$.
    Then there  is a positive  constant $C$, 
   depending only on  $c_0$ and  $(b_1,\cdots, b_n)$,
such that the following holds.
For any hyperbolic surface $M$ of genus zero with
geodesic boundary components of lengths
$b_1,\cdots, b_n$ and short geodesics $\Gamma$, 
$$
     C^{-1}\prod_{\gamma\in\Gamma}\ell(\gamma)\leq \I(M) \ .
$$
\end{theorem}

\section{Further results}

\subsection{Properness and Steklov isospectral surfaces}
Now we provide proofs of the other consequences of Theorems \ref{thm:main}
and \ref{thm:asymptotics}.

    \begin{proof}[Proof of Theorem \ref{thm:proper}]
        Let $\{M_j\}$ be a sequence  of genus zero hyperbolic surfaces
        with geodesic boundaries of fixed lengths $b_1, \ldots, b_n$.
        After passing to a subsequence, we may assume
        there is a nonempty collection 
        $\Gamma_j$  of geodesics all of  whose lengths $\ell(\gamma)\to 0$ as $j\to
        \infty$, and all other geodesics have lengths bounded away from
        zero. 
        Since the  $M_j$ have genus zero, 
         In this case, 
         $(\#\Gamma+1)\prod \omega_{M_j}(\gamma)={\det}^\ast
         \Delta_{(G_\Gamma, \omega_{M_j})}$ (see \eqref{eqn:mtt*}), and each $\omega_{M_j}(\gamma)$ is
         comparable to the length $\ell_{M_j}(\gamma)$.
        We can then use \eqref{eqn:asymptotics}
        and Corollary \ref{cor:estimate} below
         to conclude that
        \begin{equation} \label{eqn:I-estimate} 
        \I(M_j)\leq C\,        \max\left\{\ell(\gamma_1)\cdots \ell(\gamma_{k-1})
        \mid \gamma_1,\ldots,\gamma_{k-1}\in\Gamma_j\text{ distinct}\right\}
        \ ,
        \end{equation} 
        where $k$ is the number of nonisolated components of $(M_j)_{\Gamma_j}$ that intersect
        $\partial M_j$. We may assume $k$ is constant  and $C$ is independent of $j$. 
        Then \eqref{eqn:I-estimate} implies that
        \begin{equation} \label{eqn:H-estimate} 
        \Hcal(M_j)\geq (k-1)\min_{\gamma\in \Gamma_j}\log(1/\ell_j(\gamma))-\log
        C\ .
        \end{equation} 
        For a  connected tree with more than one vertex, there are at least two
        vertices having  only a single edge. 
        Since the  components of  $(M_j)_{\Gamma_j}$ must have at least
        $3$ boundary components, this implies  $k\geq 2$.
        Since $\ell_j(\gamma)\to 0$, \eqref{eqn:H-estimate} implies that
        $\Hcal(M_j)$ is unbounded along $\{M_j\}$.
        
        If $g\neq 0$, then we may find a family of surfaces $M_\varepsilon$
        with a geodesic $\gamma_\varepsilon$ of length $\ell(\varepsilon)\sim
        1/\log(1/\varepsilon)$, such that
        $M_\varepsilon\setminus\gamma_\varepsilon$ consists of two
        components: one component $M'_\varepsilon$
        containing all components of the boundary $\partial
        M(\varepsilon)$, and an isolated component $N_\varepsilon$
        obtained by removing a disk of radius $\varepsilon$ from a genus
        one Riemann surface $N$.
        As in the proof above, we may choose an annulus $A_\varepsilon$
        about $\gamma_\varepsilon$ whose boundary lengths are bounded above
        and below. Now apply Theorem \ref{thm:main} to the case where
        $\Gamma=\partial A_\varepsilon$. 
        Then the Neumann jump operators $\Ncal$ and $\Ncal_A$ both have a
        small eigenvalue $\sim 1/\log(1/\varepsilon)$ corresponding to the
        constant function $1$  on the component of $\partial A_\varepsilon$
        meeting $N_\varepsilon$, and $0$ on the other
        component of $\partial A_\varepsilon$. 
        As in the proof above,
        orthogonal to this space, the operators converge up to trace-class
        to the corresponding operators on $N$ and  the surface 
        $M'_\varepsilon$ union a disk.
        The small eigenvalue cancels the vanishing of
        $\I(A_\varepsilon)$ in the gluing formula, with the remaining
        factors bounded. Hence, $\Hcal(M_\varepsilon)$ remains bounded as
        $\varepsilon\to 0$. 
    \end{proof}

\begin{proof}[Proof of Corollary \ref{cor:steklov}]
    By \cite[Theorem 1.7]{GPPS:14}, the lengths of the boundary components
    of all the members of $\Fcal$ are equal to some fixed lengths
    $(b_1,\ldots, b_n)$.  From Theorem \ref{thm:proper},
    $\Fcal$ is contained in a compact subset of the moduli space
    $\Mbold(0;b_1,\ldots, b_n)$. 
    The result then follows as in \cite{OPS:89}.
\end{proof}

\subsection{Dirichlet and Neumann Laplacians}

\begin{corollary} \label{cor:asymptotics}
Consider the situation in Theorem \ref{thm:asymptotics}.
Let $\{\kappa_i\}$ be the collection of small eigenvalues, as in the
    Introduction.
   Then
the constant $C$ may be chosen such that
    for any hyperbolic surface $M$ of genus zero with
geodesic boundary components of lengths
$b_1,\cdots, b_n$, 
    \begin{align*}
    C^{-1} \prod_{\gamma\in \Gamma}
        \exp(-\pi^2/3\ell(\gamma))&\ell^{-3/2}(\gamma)\left(
         \det(\Delta_{(G_\Gamma,\omega_M)}+D)\prod \kappa_i\right)^{1/2}
        \leq [\Det\Delta_D]_M\\
        &\leq C \prod_{\gamma\in \Gamma}
    \exp(-\pi^2/3\ell(\gamma))\ell^{-3/2}(\gamma)\left(
                \det(\Delta_{(G_\Gamma,\omega_M)}+D)\prod \kappa_i \right)^{1/2}
                \ ,
    \end{align*}
and
    \begin{align*}
    C^{-1} \prod_{\gamma\in \Gamma}
        \exp(-\pi^2/3\ell(\gamma))&\ell^{-1/2}(\gamma)\left(\frac{\prod
        \kappa_i}{\det(\Delta_{(G_\Gamma,\omega_M)}+D)}\right)^{1/2}
    \leq [\Det^\ast\Delta_N]_M
    \\
        &\leq
    C \prod_{\gamma\in \Gamma}
        \exp(-\pi^2/3\ell(\gamma))\ell^{-1/2}(\gamma)
        \left(\frac{\prod
        \kappa_i}{\det(\Delta_{(G_\Gamma,\omega_M)}+D)}\right)^{1/2}
        \ .
    \end{align*}
    For $g\geq 1$ and $n\geq 1$, we have
$$
 [\Det\Delta_D]_M
        \leq C \prod_{\gamma\in \Gamma}
    \exp(-\pi^2/3\ell(\gamma))\ell^{-3/2}(\gamma)\left(
                \prod \kappa_i \right)^{1/2}
                \ .
$$
\end{corollary}

\begin{proof}
Let $\widehat M$ be the double of $M$. Then 
    decomposing the spectrum with respect to the 
    isometric involution,
    the small eigenvalues for the
 Laplacian on the closed surface $\widehat M$ are exactly the collection
    $\{\kappa_i\}$. Moreover, since the boundary lengths of $M$ are fixed,
    we may ignore them in the asymptotics. Hence, the short geodesics
    of $\widehat M$ correspond to the short geodesics in $M$ and their
    mirrors in the double. 
    By \cite[Theorem 5.3]{Wolpert:87} there is a constant $B>1$ such that
    \begin{equation} \label{eqn:wolpert} 
        B^{-1}\leq \frac{[\Det^\ast\Delta]_{\widehat M}}{\prod_{\gamma\in
    \Gamma}\exp(-2\pi^2/3\ell(\gamma))\ell^{-2}(\gamma)\prod_i\kappa_i}
    \leq B \ .
    \end{equation} 
    On the other hand, from \eqref{eqn:weisberger} and \eqref{eqn:double}, we have
$$
        [\Det\Delta_D]_M^2=\frac{[\Det^\ast\Delta]_{\widehat M}}{A(M)\I(M)}
        \quad ,\quad
        [\Det\Delta_N]_M^2=[\Det^\ast\Delta]_{\widehat M}A(M)\I(M)
        \ .
$$ 
    The result now follows from \eqref{eqn:wolpert} and Theorems
    \ref{thm:asymptotics} and \ref{thm:higher-genus}.
\end{proof}

\section{Graph Laplacians} \label{sec:graphs}
\subsection{Matrix tree theorem with potential}
For the proof of Theorem \ref{thm:proper}
we require the results in this section,  perhaps well-known, but for which we have been
unable to locate precise statements in the vast literature on this subject.
 For the sake of
completeness, we therefore provide proofs here.
This will also allow us to review the construction and basic facts of graph
Laplacians. 
The main result, Theorem \ref{thm:kirchhoff}, is an extension of the
weighted matrix tree theorem of Kirchhoff for the graph Laplacian with an
added diagonal potential. 
Corollary \ref{cor:estimate} then gives  a comparison of the
determinants of the graph Laplacians with and without the 
potential.

Let  $G$ be an undirected graph with vertex and edge sets $V(G)$ and
$E(G)$,
respectively. Label the elements of $V(G)$ by $v_i\in V$, $i=1,\ldots, n$.  
For $i\neq j$ we say $(ij)\in E(G)$ if there is an edge between $v_i$ and $v_j$. 
We always assume $G$ is \emph{simple},  by which we mean there is at most one edge
between distinct vertices, and no edge from a vertex to itself. 
A weight function on $G$ is a map $\omega:E(G)\to \RBbb$. The weight defines
(and is determined by) an
associated $n\times n$ symmetric matrix:
$$\omega_{ij}:=\begin{cases} \omega(ij)& \text{ if }(ij)\in E\ ,\\
0 & \text{otherwise.}\end{cases}
$$
If we set $\displaystyle \mu_i=\sum_{(ij)\in E}\omega_{ij}$, then
the (weighted) graph Laplacian is the $n\times n$ symmetric matrix: 
\begin{equation} \label{eqn:laplacian}
(\Delta_{(G,\omega)})_{ij}=\begin{cases} 
    -\omega_{ij} &  (ij)\in E\ , \\
    \mu_i & i=j \ , \\
    0& \text{otherwise.}
\end{cases}
\end{equation}
The weight $\omega$ is positive if $\omega(ij)>0$ for all $(ij)\in E(G)$. 
When the weights are positive, the  matrix $\Delta_{(G,\omega)}$ is positive semidefinite with a zero eigenvalue of
multiplicity $1$ if $G$ is connected. 
We let ${\det}^\ast\Delta_{(G,\omega)}$ denote the product of the nonzero eigenvalues.
By a \emph{potential} we mean a function $\delta: V(G)\to \RBbb$. If
$\delta_i=\delta(v_i)$, then $\delta$ is represented by a diagonal matrix with
entries $\delta_i$, which we will typically denote by $D$. 
Given a potential, we shall say a vertex $v$ is \emph{marked} if 
$\delta(v)\neq 0$. The potential is positive if $\delta(v)$ is either zero
or positive for every $v$. 

For a connected graph $G$, let $\Sp(G)$ be the set of spanning trees of $G$, i.e.\
connected trees $T\subset G$ such that $V(T)=V(G)$.
For a tree $T$ with $\ell$ marked vertices $v_1,\ldots, v_\ell$, let
$\Escr(T;v_1,\ldots,v_\ell)$ denote the set of  collections of $(\ell-1)$
edges $e_1,\ldots, e_{\ell-1}\in E(T)$ such that each of the $\ell$ connected components
of $T\setminus e_1\cup\cdots\cup e_{\ell-1}$ contains exactly one marked vertex
$v_j$. 
Finally, for $T\in \Sp(G)$ and $S\in \Escr(T;v_1,\ldots,v_\ell)$, we define
a \emph{multiplicity}:
$$
m(T,S)=\#\left\{ (T',S')\mid T'\in \Sp(G)\ ,\ S'\in
\Escr(T';v_1,\ldots,v_\ell) \ ,\ 
T'\setminus S'=T\setminus S\right\}
$$
For elements of the set above, we shall say that $(T',S')$ \emph{is
equivalent  to} $(T,S)$.
With this understood,  we are ready to state the main result.

\begin{theorem} \label{thm:kirchhoff} 
    Let $(G,\omega)$ be a connected weighted graph  with $n$ vertices
    $v_1,\ldots, v_n$,  and let 
    $\Delta_{(G,\omega)}$ be the graph
    Laplacian. Fix
    a potential  $\delta:V(G)\to \RBbb$, $\delta_i=\delta(v_i)$ with
    associated diagonal matrix $D$. 
     Then
    \begin{equation} \label{eqn:egregium}
        \det(\Delta_{(G,\omega)} + D)
        =\sum_{T\in \Sp(G)}\sum_{\ell=1}^n\sum_{1\leq i_1<\cdots<i_\ell\leq
        n}\sum_{S\in
        \Escr(T;v_{i_1},\ldots,v_{i_\ell})}\frac{\delta_{i_1}\cdots\delta_{i_\ell}}{m(T,S)}\prod_{e\in
        E(T)\setminus S} \omega(e)
        \ .
    \end{equation}
\end{theorem}
Theorem \ref{thm:kirchhoff} will be proved in the next section.
First,  let us draw some conclusions.
An immediate consequence of \eqref{eqn:egregium} is the following
important
\begin{corollary} \label{cor:increasing}
   % The determinant $\det(\Delta_{(G,\omega)}+D)$ is increasing with each $\delta_i$.
    Suppose $(G,\omega)$ is a connected graph with positive weights and a
    positive potential $\delta$. 
    For $\kappa\geq 1$,  there is $C\geq 1$ depending only on
    $\kappa$, $G$, and $\delta$, such that the following holds.
For any  weight function $\widetilde \omega$ on $G$ with 
    \begin{equation} \label{eqn:kappa}
    \kappa^{-1}\omega(e)\leq \widetilde\omega(e)\leq
    \kappa\,\omega(e)
    \end{equation}
    for all $e\in E(G)$, we have
$$
    C^{-1}\det(\Delta_{(G,\omega)}+D)
    \leq \det(\Delta_{(G,\widetilde\omega)}+D)
    \leq
    C\det(\Delta_{(G,\omega)}+D)
    \ .
$$
\end{corollary}

In the following, we suppose $\delta$ has exactly $k$ nonzero entries
$\delta_1, \ldots, \delta_k$ at $v_1,\ldots, v_k$, $1\leq k\leq n$. 
Suppose first that $k=1$.
Notice that in this case, no edges are removed: in the expression
\eqref{eqn:egregium} the sum over $S$ (and therefore also the
multiplicities) is  
absent.
For $\varepsilon>0$,
    \begin{align*}
        \frac{d}{d\varepsilon}\log\det(\Delta_{(G,\omega)}+\varepsilon D)
        &=\tr\left((\Delta_{(G,\omega)}+
        \varepsilon D)^{-1}D\right)
        =\delta_1(\Delta_{(G,\omega)}+\varepsilon D)^{-1}_{11} \notag\ ; \\
        \frac{d}{d\varepsilon}\det(\Delta_{(G,\omega)}+\varepsilon D)
        &=\delta_1 \det
        ((\Delta_{(G,\omega)}+D)^{[1]})
        =\delta_1\det(\Delta_{(G,\omega)}^{[1]}) \ .
    \end{align*}
Here, we have introduced the following notation: if $A=(A_{ij})$ is an
$n\times n$ matrix, we denote by $A^{[k]}$
the $(n-1)\times(n-1)$ matrix obtained by deleting the $k$-th row and
the $k$-th column.
    Since $\det\Delta_{(G,\omega)}=0$, by integration we get
   \begin{equation} \label{eqn:red-det} 
    \det(\Delta_{(G,\omega)}+D)
        =\delta_1\det(\Delta_{(G,\omega)}^{[1]}) \ .
   \end{equation} 
    Now by the weighted matrix tree theorem (cf.\ \cite[Thm.\
    VI.27]{Tutte:01}), 
   \begin{equation} \label{eqn:mtt} 
       \det(\Delta_{(G,\omega)}^{[1]})
=\sum_{T\in \Sp(G)}\prod_{e\in E(T)} \omega(e)
\ ,
   \end{equation} 
    and so we obtain \eqref{eqn:egregium}.
    In the proof of Theorem \ref{thm:kirchhoff} below, we do not reprove 
    \eqref{eqn:mtt} but rather use it as a starting point for an inductive
    argument.

    For $k\geq 2$,
the appearance of the multiplicity $m(T,E)$ is a new feature in
the generalized matrix tree expression \eqref{eqn:egregium}. Its necessity
is immediate from the $\delta^n$ term in case $k=n$, $\delta_1=\cdots=\delta_n=\delta$.
More illuminating is   the simple
example in Figure 3.
\begin{figure}
    \begin{tikzpicture}
        \draw (0,.08) -- (2,1);
        \draw (0,-.07) -- (2,-1);
        \draw  (2,1) -- (4,.08);
        \draw (2,-1) -- (4,-.07);
        \node at (0,0) {$\circ$};
        \node at (2,1) {$\bullet$};
        \node at (2,-1) {$\bullet$};
        \node at (4,0) {$\circ$};
        \node at (-.3,0) {$v_1$};
        \node at (4.3,0) {$v_2$};
        \node at (1.6,1.2) {$v_3$};
        \node at (2.5,-1.2) {$v_4$};
    \end{tikzpicture}
    \caption{}
\end{figure}
Here, the weights are $\delta(v_i)=\delta_i$, $i=1,2$ and zero otherwise.
Then one calculates the $\delta_1\delta_2$ term directly:
\begin{equation} \label{eqn:cool-example}
\det(\Delta_{(G,\omega)}+D)=\delta_1\delta_2(\omega_{13}+\omega_{23})(\omega_{14}+\omega_{24})
+ \cdots
\end{equation}
There are $4$ spanning trees for $G$, obtained by removing a single edge.
For each tree $T$, there are two possible edges $S$ that can be removed to separate
$v_1$ from $v_2$.  Thus there are $8$ terms in the $\delta_1\delta_2$ sum
in \eqref{eqn:egregium}. But the
multiplicity of each pair $(T,S)$ is clearly $2$, corresponding to switching the edge removed to
define the tree $T$ with the edge $S$ removed from $T$. The $8$ terms thus
reduce to the $4$ terms in \eqref{eqn:cool-example}.

    A second special case is where $G$ is a tree. 
For  $\delta_i>0$, $i=1,\ldots, k$, and zero otherwise, from 
    \eqref{eqn:egregium} we have 
    \begin{equation} \label{eqn:min} 
\det(\Delta_{(G,\omega)}+D)\geq
\delta_1\cdots\delta_k\min\left\{\omega(e_1)\cdots\omega(e_{n-k}) \mid
e_1,\ldots, e_{n-k}\in E(G)\ \text{distinct}\right\}
   \end{equation} 
    We are mostly interested in the case where  the edge weights are much
    smaller than the $\delta_i$'s. The estimate above can probably be
    improved. However, notice that in the example 
    \eqref{eqn:cool-example}, $\omega_{13}\omega_{23}$ (or
    $\omega_{14}\omega_{24}$) do not appear in the $\delta_1\delta_2$ term. 
    If $\omega_{14}$ and $\omega_{24}$ are big compared to the other two
    weights, we cannot replace $\min$ by $\max$ in \eqref{eqn:min}.

    For a connected graph,  the weighted matrix tree theorem
    (the equality \eqref{eqn:mtt}, which holds for any principal minor)
    implies,
    \begin{equation} \label{eqn:mtt*}
        {\det}^\ast\Delta_{(G,\omega)}
=n\sum_{T\in \Sp(G)}\prod_{e\in E(T)} \omega(e) \ .
    \end{equation}
   In case   $G$ is a tree, there is only one term in the sum. 
     Hence, from \eqref{eqn:red-det}, \eqref{eqn:mtt}, and  \eqref{eqn:min} we obtain
\begin{corollary} \label{cor:estimate}
    Let $(G,\omega)$ be a weighted tree with $n$ vertices, and suppose
 $D$ has exactly $k\geq 1$ nonzero entries $\delta_1,\ldots,\delta_k>0$. 
    Then  if $k=1$,
    $$
    {\det}^\ast\Delta_{(G,\omega)}=\frac{n}{\delta_1}
\det(\Delta_{(G,\omega)}+D) \ ,
$$
and if $k\geq 2$,
$$
    \frac{{\det}^\ast(\Delta_{(G,\omega)})}{\det(\Delta_{(G,\omega)}+D)}\leq
        \frac{n}{\delta_1\cdots \delta_k}\,
        \max\left\{\omega(e_1)\cdots\omega(e_{k-1}) \mid
e_1,\ldots, e_{k-1}\in E(G)\ \text{distinct}\right\} \ .
$$
\end{corollary}

Finally, 
an important technical result for this paper is the following, which is used in
Section \ref{sec:asymptotics}. Let $(G,\omega)$ be a connected, weighted
tree with $n$ vertices and positive weights. 
%Let $\delta : V(G)\to \RBbb^+$ be a marking of $k$ vertices, $1\leq k\leq n$. 
Let $\widehat G$  be 
  the complete graph on the vertices of $G$.
Fix $\kappa_0>0$.
Suppose $\widehat \omega$ is a system of weights (not necessarily
positive) for $\widehat G$ satisfying
\begin{enumerate}
    \item $\widehat\omega(e)=\omega(e)$ for all $e\in E(G)$;
    \item For $\widehat e\in E(\widehat G)\setminus E(G)$  between vertices
        $v_1$ and $v_2$, 
        $$
        |\widehat\omega(\widehat e)|\leq \kappa_0\omega(e)
        $$
        for any $e\in E(G)$ along the geodesic in $G$ from $v_1$ to $v_2$. 
\end{enumerate}

\begin{corollary} \label{cor:comparison}
    Fix $(G,\omega)$ as above, and let  $\delta:V(G)\to \RBbb_{\geq 0}$ be a
     positive potential.
     Then for $\kappa_0>0$ sufficiently small (depending upon $(G,\omega)$
     and $\delta$), 
      there is a constant $C\geq 1$ depending
    only on $G$, $\delta$,  and $\kappa_0$,  such that if $(\widehat
    G,\widehat\omega)$ satisfies (i) and (ii)   above,
    $$
    C^{-1}\det(\Delta_{(G,\omega)}+D)\leq \det(\Delta_{(\widehat G,
    \widehat\omega)}+D) \leq C\det(\Delta_{(G,\omega)}+D)
    \ .
    $$
\end{corollary}

\begin{proof}
    We wish to compare the terms appearing in \eqref{eqn:egregium} 
    for $(G,\omega)$ and $(\widehat G, \widehat\omega)$. 
Let $\widehat T$ be a spanning tree for $\widehat G$.  
        Consider a
    component subtree $\widehat T'$ of  $\widehat T\setminus \widehat S$,
    for a separating set of edges $\widehat S$.     Suppose that  $\widehat T'$ contains an edge $\widehat e$ not in $E(G)$. 
    Let $v_1$ and $v_2$ be the  two vertices of $\widehat e$. Since
    $\widehat T'$ contains only one marked vertex, we may assume $v_1$ is
    not marked. Let $g$ be the geodesic in $G$ from $v_1$ to $v_2$.
    Because $\widehat T'$ is a tree, we cannot have $g\subset \widehat T'$,
    since then $g\cup \widehat e$ would be a cycle. 
    Hence, 
    let
    $e$ be the first edge  in $g$ (going from $v_1$ to $v_2$) that is not contained in $\widehat T'$. 
Then if we replace $\widehat e$ by $e$ we obtain a new spanning tree
$\widehat T_1$
    (with the same separating set $\widehat S$) with fewer edges  that are not in
    $E(G)$. 
    Moreover,  by (ii) the product of the edges in $\widehat T\setminus \widehat S$ is
    strictly less (in absolute value) than that of  $\widehat T_1\setminus
    \widehat S$.
    Continuing in this way, we find a new spanning tree $\widehat T_\bullet$  of $\widehat G$ 
     such that $\widehat T_\bullet\setminus \widehat S\subset G$. Now
    there is a unique separating set $S\subset E(G)$ such that $\widehat
    T_\bullet\setminus \widehat S=G\setminus S$. 
    Thus, the term in the expansion \eqref{eqn:egregium} for $\widehat G$
    corresponding to $(\widehat T, \widehat
    S)$ is dominated by the term $(G,S)$ in the expansion for $G$.  
    This completes the proof. 
\end{proof}

\subsection{Proof of Theorem \ref{thm:kirchhoff}}
The proof is by induction on $n$ and $k =$ the number of nonzero entries of
$D$. Thus, we assume \eqref{eqn:egregium} holds for graphs with fewer than
$n$ vertices and any $D$. 
We have seen in \eqref{eqn:mtt}
that by the usual weighted matrix tree theorem,  the result holds for all $n$ and $k=1$. 
Suppose now that  $k\geq 2$, and that \eqref{eqn:egregium} holds for $n$
vertices and potentials with  fewer than $k$ nonzero entries.
 We must show that for $D$ with exactly $k$ nonzero entries,
    \begin{equation} \label{eqn:egregium-k}
        \det(\Delta_{(G,\omega)} + D)
        =\sum_{T\in \Sp(G)}\sum_{\ell=1}^k\sum_{1\leq i_1<\cdots<i_\ell\leq
        k}\sum_{S\in
        \Escr(T;v_{i_1},\ldots,v_{i_\ell})}\frac{\delta_{i_1}\cdots\delta_{i_\ell}}{m(T,S)}\prod_{e\in
        E(T)\setminus S} \omega(e)
        \ .
    \end{equation}
Set $\delta_k=\delta(v_k)$.
By the same argument used above to derive \eqref{eqn:mtt}, we have
\begin{equation}  \label{eqn:induction}
    \det(\Delta_{(G,\omega)}+D)=\det(\Delta_{(G,\omega)}+D)\bigr|_{\delta_k=0}+
    \delta_k\det((\Delta_{(G,\omega)}+D)^{[k]})
    \ .
\end{equation}
We view the second term on the right hand side as the determinant of a new
weighted graph $\widetilde G$ with potential $\widetilde D$,  obtained by deleting $v_k$ and all edges at
$v_k$. The weight function  $\widetilde\omega$ is the restriction of $\omega$
to $\widetilde G$.
Importantly, since the edges of $v_k$ have been deleted, the $\mu_j$ differ
from $\widetilde\mu_j$, and 
$\widetilde D$  is determined by the rule
\begin{equation} \label{eqn:delta}
\widetilde \delta_j=\begin{cases}
    \delta_j & (jk)\not\in E(G)\ , \\
    \delta_j+\omega_{jk} & (jk)\in E(G) \ ,
\end{cases}
\end{equation}
for all $v_j\in \widetilde G$. 
With this interpretation we have
\begin{equation} \label{eqn:det-minor}
    \det((\Delta_{(G,\omega)}+D)^{[k]}) 
=
\det(\Delta_{(\widetilde G,\widetilde\omega)}+\widetilde D)
\ .
\end{equation}
By induction on $k$, we may assume the first term on the right hand side of   \eqref{eqn:induction}
satisfies \eqref{eqn:egregium}, with $\delta_k$ set to zero. 
This accounts for all the terms on the right
hand side of \eqref{eqn:egregium-k} where $i_\ell\leq k-1$. The remaining
terms all have $i_\ell=k$, and therefore a factor of $\delta_k$. 
Given \eqref{eqn:det-minor}, in order
to complete the proof we must show that 
\begin{align}
    \begin{split}\label{eqn:main-equation}
\det(\Delta_{(\widetilde G,\widetilde\omega)}+\widetilde D)
        &=\sum_{\ell=1}^k\sum_{1\leq i_1<\cdots<i_{\ell-1}\leq
        k-1}\delta_{i_1}\cdots\delta_{i_{\ell-1}}\\
        &\hspace{-1cm}\times \sum_{T\in \Sp(G)}\sum_{S\in
        \Escr(T;v_{i_1},\ldots,v_{i_{\ell-1}},v_k)}\frac{1}{m(T,S)}\prod_{e\in
        E(T)\setminus S} \omega(e)
        \ .
    \end{split}
\end{align}
In the sum above, $\ell=1$ is taken to mean that no $\delta_i$'s appear. 
Let $w_1,\ldots, w_m$ be the vertices adjacent to $v_k$, with edges $f_1,
\ldots, f_m$.  See Figure 4.
\begin{figure}
    \begin{tikzpicture}
        \draw (.03,.06) -- (2,1);
        \draw (.07,.04) -- (2,.5);
        \draw (.08,0) -- (2,0);
        \draw (.03,-.06) -- (2,-1);
        \draw[dotted] (2,-.2) -- (2,-.8);
        \node at (0,0) {$\circ$};
        \node at (2,1) {$\bullet$};
        \node at (2,.5) {$\bullet$};
        \node at (2,0) {$\bullet$};
        \node at (2,-1) {$\bullet$};
        \node at (4,1) {$\circ$};
        \node at (4,0) {$\circ$};
        \node at (4,-1) {$\circ$};
        \draw (2,1) -- (2.5,1);
        \draw (2,.5) -- (2.5,.5);
        \draw (2,0) -- (2.5,0);
        \draw (2,-1) -- (2.5,-1);
        \draw[dotted] (2.5,1) -- (3,1);
        \draw[dotted] (2.5,.5) -- (3,.5);
        \draw[dotted] (2.5,0) -- (3,0);
        \draw[dotted] (2.5,-1) -- (3,-1);
        \node at (-.5,0) {$v_k$};
        \node at (4.5,1) {$v_1$};
        \node at (4.5,0) {$v_2$};
        \node at (4.8,-1) {$v_{k-1}$};
        \node at (2.3,1.2) {$w_1$};
        \node at (2.3,-1.3) {$w_m$};
        \node at (1,.8) {$f_1$};
        \node at (1,-.8) {$f_m$};
    \end{tikzpicture}
    \caption{}
\end{figure}
Let us first assume that
\begin{enumerate}
    \item $G\setminus\{v_k\}$ \emph{is connected};
    \item \emph{None of the $w_j$'s are marked in $G$}.
\end{enumerate}
Thus, by \eqref{eqn:delta}, $\widetilde G$ has $(k+m-1)$ marked points at
$v_1,\ldots, v_{k-1}$ and $w_1,\ldots, w_m$.  
Since $\widetilde G$ has $(n-1)$ vertices and we have assumed the result
holds in this case for all $k$,
by induction  we have
\begin{align}
    \begin{split} \label{eqn:hat}
\det(\Delta_{(\widetilde G,\widetilde\omega)}+\widetilde D)
        &=\sum_{\substack{\ell=1,\ldots, k\\ \ell'=0,\ldots,
        m}}\sum_{\substack{1\leq i_1<\cdots<i_{\ell-1}\leq k-1\\ 
        1\leq j_1<\cdots<j_{\ell'}\leq m}}
        \delta_{i_1}\cdots\delta_{i_{\ell-1}}
                \\
        &\hspace{-1cm}\times 
        \sum_{\widetilde T\in \Sp(\widetilde G)}\sum_{\widetilde S\in
        \Escr(\widetilde T;v_{i_1},\ldots,v_{i_{\ell-1}}, w_{j_1}, \ldots,
        w_{j_{\ell'}})}\frac{\widetilde\delta_{j_1}\cdots\widetilde\delta_{j_{\ell'}}
}{m(\widetilde T,\widetilde S)}\prod_{e\in
        E(\widetilde T)\setminus \widetilde S} \omega(e)
        \ ,
    \end{split}
\end{align}
where by $\ell'=0$ we mean no $\widetilde\delta_j$'s appear, and in the sum
we do not allow both $\ell=1$ and $\ell'=0$. 
In order to prove the equality of the right hand sides of \eqref{eqn:main-equation} and
\eqref{eqn:hat}, 
for a fixed  choice of $i_1,\ldots, i_{\ell-1}$, we must
find a correspondence between trees and edge sets in $\widetilde G$ and $G$,
modulo equivalences.
\medskip

\noindent{\bf Case 1}.
Suppose first that $\ell'\geq 1$. 
Let $\widetilde T\in \Sp(\widetilde G)$, $\widetilde S\in \Escr(\widetilde
T;i_1,\ldots, i_{\ell-1}, j_1,\ldots, j_{\ell'})$, 
so that $\# \widetilde S=\ell+\ell'-2$.
%Let $\tilde j_1, \ldots, \tilde j_{m-\ell'}$ denote the complement to
%$\{j_1,\ldots, j_{\ell'}\}$ in $\{1,\ldots, m\}$. 
Let $\widetilde e_1,\ldots, \widetilde e_{\ell'-1}\in \widetilde S$ be the edges
that separate $w_{j_1},\ldots, w_{j_{\ell'}}$. To be precise, there is a
unique geodesic $g_{12}$ in $\widetilde T$ from $w_{j_1}$ to $w_{j_2}$, and by
definition of the set $\Escr$  there is an edge in $\widetilde S$ that is a
segment of $g_{12}$.
 Choose $\widetilde e_1$ to be the nearest such edge to $w_{j_1}$ in
the simplicial metric. Now consider the geodesic $g_{23}$ from $w_{j_2}$ to
$w_{j_3}$. This may or may not be separated by $\widetilde e_1$. If it is,
then $g_{23}$ intersects the geodesic $g_{13}$ from $w_{j_1}$ to $w_{j_3}$. We then
choose $\widetilde e_2$ to be the nearest edge to $w_{j_1}$  along this
geodesic. If $g_{23}$ is not separated, choose $\widetilde e_2$ to be
the nearest edge to $w_{j_2}$. Continuing in this way, we determine the
collection
$\widetilde e_1,\ldots, \widetilde e_{\ell'-1}\in \widetilde S$. Now define
$T\subset G$ by
$$
T=\left(\widetilde T\setminus \widetilde e_1\cup\cdots\cup \widetilde e_{\ell'-1}
\right)\cup f_{j_1}\cup\cdots \cup f_{j_{\ell'}} \cup\{v_k\}\ .
$$
We claim that $T$ is a connected tree. Being a subset of $\widetilde T$, 
$T\cap\widetilde G$
is a tree. By construction, $w_{j_1},\ldots, w_{j_{\ell'-1}}$ are in
distinct components of $T\cap\widetilde G$. It follows that $T$ is a tree as well. 
That $T$ is connected follows from the fact that the connected components
of $\widetilde T\setminus \widetilde e_1\cup\cdots\cup \widetilde e_{\ell'-1}$
are in 1-1 correspondence with the $\{w_{j_i}\}$. 
Now 
the remaining $\ell-1$ edges in $\widetilde S$ -- let us denote them
$e_1,\ldots, e_{\ell-1}$ -- provide an element $S\in
\Escr(T; i_1,\ldots, i_{\ell-1}, v_k)$. 
Indeed, removing $e_1,\ldots, e_{\ell-1}$ separates the 
$v_{i_1},\ldots, v_{i_{\ell -1}}$ 
from all the $w_{j_1},\ldots,w_{j_{\ell'-1}}$, and further removing
$\widetilde e_1\cup\cdots\cup\widetilde e_{\ell'-1}$ separates the $w_{j_1},\ldots,w_{j_{\ell'-1}}$ from themselves in $\widetilde T$. 
It follows that in $T$, $v_k$ is separated from 
the $v_{i_1},\ldots, v_{i_{\ell -1}}$ 
as well. It is clear that equivalent pairs $(\widetilde
T,\widetilde S)$ give equivalent counterparts $(T,S)$. 
Indeed, 
$$
T\setminus S=(\widetilde T\setminus \widetilde S)\cup f_{j_1}\cup\cdots \cup f_{j_{\ell'}}
\cup\{v_k\}\ .
$$
This construction may be reversed.
Starting from the pair $(T,S)$, we construct  $(\widetilde T,\widetilde S)$ as
follows. Let $f_{j_1}, \ldots, f_{j_{\ell'}}$ be all the  edges in $T\setminus S$
from  $v_k$. 
The first step is to replace $(T,S)$ with an equivalent pair $(T',S')$
so that $f_{j_1}, \ldots, f_{j_{\ell'}}$ are the only edges from $v_k$ in
$T'$.  Let $f_p\in S$ be another such edge, to $w_p$.
Let $e\in E(\widetilde G)$ be the edge realizing the minimal distance from
the component of $T\cap \widetilde G$ containing $w_p$ to the other components. 
Then if we let $T'=(T\setminus f_p)\cup e$, $S'=(S\setminus \{f_p\})\cup \{e\}$,  
then clearly $T'$ is a tree. 
Hence, we may assume, up to equivalence, that the edges in $T$ from $v_k$ are not
in $S$. 
Now the components of $T\cap \widetilde G$ are in 1-1 correspondence with the
$w_{i_j}$.  Let $\widetilde e_1, \ldots , \widetilde e_{\ell'-1}$ be edges in
$\widetilde G$ minimizing the distances between these components. 
%Now  $T\cap \widetilde G$ has $\ell$ components,
%and since $T$ is a tree, the   $w_{j_i}$ lie in distinct components. 
%We have assumed $\widetilde G$ is connected. Let $g$ be a path from
%$w_{j_1}$ to $w_{j_2}$ in $\widetilde G$, and let $e\subset g$ be the first
%edge of $g$  not in
%$T$. Then $e\not\in S$,  and so $e$ must join a vertex in the connected
%component of $w_{j_1}$ with a vertex in the connected component of some
%other $w_{j_i}$. After relabelling, assume this component is $w_{j_2}$, and
%call the edge $e=\widetilde e_1$. Next, we take a path from $w_{j_2}$ to
%$w_{j_3}$ and consider the first edge not in $T\cup \widetilde e_1$.
%Proceeding in this way, we find edges $\widetilde e_1, \ldots, \widetilde
%e_{\ell'-1}$.  
We set 
\begin{align*}
    \widetilde T&=(T\cap \widetilde G)\cup \widetilde
e_1\cup \cdots \cup \widetilde e_{\ell'-1}
\ , \\
    \widetilde S&= \left\{ \widetilde e_1, \ldots,  \widetilde e_{\ell'-1}, e_1,
\ldots, e_{\ell-1}\right\} \ .
\end{align*}
Then the pair $(\widetilde T,\widetilde S)$  is the desired inverse, modulo
equivalence of the previous construction. 
Finally, notice from \eqref{eqn:delta}  that in this construction,
$\widetilde\delta_{j_i}=\omega(f_{j_i})$.  Hence,
$$
\prod_{e\in T\setminus S}\omega(e)=\widetilde\delta_{j_1}\cdots
\widetilde\delta_{j_{\ell'}}\prod_{e\in \widetilde T\setminus\widetilde S}\omega(e)
\ .
$$
We  have therefore found a correspondence of terms in \eqref{eqn:main-equation} 
with some $f_j\in T\setminus S$, and terms in 
\eqref{eqn:hat} with $\ell'\geq 1$. As seen above, equivalent pairs $(\widetilde
T,\widetilde S)$ give equivalent counterparts $(T,S)$. 
\medskip

\noindent{\bf Case 2}. Now suppose $\ell'=0$, i.e.\  none of the
points $w_1,\ldots, w_m$ are marked. 
Note that by our rule  this forces $\ell\geq 2$.
Let $\widetilde T\in \Sp(\widetilde G)$, $\widetilde S\in \Escr(\widetilde
T;i_1,\ldots, i_{\ell-1})$, 
so that now $\# \widetilde S=\ell-2$.
If nonempty, 
enumerate the elements of $\widetilde S$ by $\widetilde e_1,\ldots, \widetilde
e_{\ell-2}$. Write a disjoint union
$$
\{w_1,\ldots, w_m\}=\Ccal_1\sqcup \cdots\sqcup \Ccal_q
\ ,
$$
so that each $\Ccal_i$ lies in a distinct  connected component of $\widetilde T\setminus \widetilde
e_1\cup\cdots\cup \widetilde e_{\ell-2}$. 
As in the previous case, we may choose a subset of $\widetilde S$, which
after renumbering we
assume to be $\widetilde e_1,\ldots, \widetilde e_{q-1}$,
so that each $\Ccal_i$ lies in a distinct  connected component of $\widetilde T\setminus \widetilde
e_1\cup\cdots\cup \widetilde e_{q-1}$. 
For each $\Ccal_i$, choose an edge $f_{j_i}$ from one of the elements of
$\Ccal_i$ to $v_k$. 
We obtain a tree $T\in\Sp(G)$ by adding the edges $f_{j_1},\ldots,
f_{j_q}$ to $\widetilde T$, and deleting $\widetilde e_1,\ldots ,\widetilde e_{q-1}$.
The new edge set is
$$
S=\{f_{j_1},\ldots, f_{j_q}, \widetilde e_q, \ldots, \widetilde e_{\ell-2}\}\in
\Escr(T; v_{i_1}, \ldots, v_{i_{\ell-1}}, v_k)
\ .
$$
Clearly, the choice of $f_{j_i}$'s give equivalent pairs $(T,S)$.
Similarly, equivalent pairs $(\widetilde T,\widetilde S)$ give equivalent pairs
$(T,S)$. 
In this case, $T\setminus S=(\widetilde T\setminus \widetilde S)\cup\{v_k\}$.
Going the other way,
suppose the edges from $v_k$ of a spanning tree $T$ are $f_{j_1},\ldots,
f_{j_q}$, and that these are all contained  in an edge set $S$. 
Let $\widetilde e_q,\ldots, \widetilde e_{\ell-2}$ denote the remaining
edges in $S$.
Then $T\cap \widetilde G$ has exactly $q$ connected components in 1-1
correspondence with the $w_{j_i}$. 
Find $\widetilde e_1, \ldots,
\widetilde e_{q-1}$ in $\widetilde G$ connecting these components of
$w_{i_j}$'s,  in a manner exactly the same as in Case 1. We then set
\begin{align*}
    \widetilde T&=(T\cap \widetilde G)\cup \widetilde e_1\cup \cdots\cup
    \widetilde e_{q-1}\ , \\
    \widetilde S&=\{\widetilde e_1,\ldots, \widetilde e_{\ell-2}\}\ .
\end{align*}
This is inverse to the previous construction. Thus, we have  a  correspondence between terms in \eqref{eqn:main-equation} 
with  $f_j\not\in T\setminus S$, $j=1,\ldots, m$,  and terms in 
\eqref{eqn:hat} with no $\widetilde \delta_j$'s. 

We now address assumptions (i) and (ii). Suppose that $G\setminus\{v_k\}$
is not connected. Notice that the right hand side of 
\eqref{eqn:hat} is multiplicative and that extending spanning trees of
each component of $G\setminus\{v_k\}$ to include $v_k$ uniquely determines
a spanning tree of $G$. Hence, since  the analysis above applies to each component assumption (i) may be dropped. For (ii), if
one of the points, e.g.\ $w_1$, is marked in $G$, then after relabelling we
may assume $w_1=v_{k-1}$. Then
$$
\widetilde\delta(w_1)=\delta(v_{k-1})+\omega(f_1)\ ,
$$
and $\widetilde G$ has $k+m-2$ marked points. In the expression
\eqref{eqn:hat}, terms involving $\widetilde \delta(w_1)$ split into terms
with  $\delta(v_{k-1})$ and those with $\omega(f_1)$.
The latter correspond  to terms in 
\eqref{eqn:main-equation} with $v_{k-1}$ unmarked, just as in the cases
considered above. The terms involving $\delta(v_{k-1})$
correspond  to terms in 
\eqref{eqn:main-equation} with $v_{k-1}$ marked. In this case, in the
definition of $T$, we include $f_1$ in the set $S$, but otherwise proceed
as above. 
This completes the proof.

\section*{Acknowledgements} The author is grateful to the editors for
putting together this special volume. He  also thanks Jean Lagac\'e
and David Sher for their comments. 

$\hbox{}$

\bibliographystyle{plain}

\bibliography{../papers}

\providecommand{\MR}[1]{}
\begin{thebibliography}{10}

\bibitem{Alvarez:83}
Orlando Alvarez.
\newblock Theory of strings with boundaries: fluctuations, topology and quantum
  geometry.
\newblock {\em Nuclear Phys. B}, 216(1):125--184, 1983.

\bibitem{Burger:90}
Marc Burger.
\newblock Small eigenvalues of {R}iemann surfaces and graphs.
\newblock {\em Math. Z.}, 205(3):395--420, 1990.

\bibitem{BFK:92}
D.~Burghelea, L.~Friedlander, and T.~Kappeler.
\newblock Meyer-{V}ietoris type formula for determinants of elliptic
  differential operators.
\newblock {\em J. Funct. Anal.}, 107(1):34--65, 1992.

\bibitem{SteklovSurvey}
Bruno Colbois, Alexandre Girouard, Carolyn Gordon, and David Sher.
\newblock Some recent developments on the {S}teklov eigenvalue problem.
\newblock 2022.
\newblock In preparation.

\bibitem{DPRS:87}
Jozef Dodziuk, Thea Pignataro, Burton Randol, and Dennis Sullivan.
\newblock Estimating small eigenvalues of {R}iemann surfaces.
\newblock In {\em The legacy of {S}onya {K}ovalevskaya ({C}ambridge, {M}ass.,
  and {A}mherst, {M}ass., 1985)}, volume~64 of {\em Contemp. Math.}, pages
  93--121. Amer. Math. Soc., Providence, RI, 1987.

\bibitem{Edward:93}
Julian Edward.
\newblock Pre-compactness of isospectral sets for the {N}eumann operator on
  planar domains.
\newblock {\em Comm. Partial Differential Equations}, 18(7-8):1249--1270, 1993.

\bibitem{EdwardWu:91}
Julian Edward and Siye Wu.
\newblock Determinant of the {N}eumann operator on smooth {J}ordan curves.
\newblock {\em Proc. Amer. Math. Soc.}, 111(2):357--363, 1991.

\bibitem{Forman:87}
Robin Forman.
\newblock Functional determinants and geometry.
\newblock {\em Invent. Math.}, 88(3):447--493, 1987.

\bibitem{FriedlanderGuillemin:08}
Leonid Friedlander and Victor Guillemin.
\newblock Determinants of zeroth order operators.
\newblock {\em J. Differential Geom.}, 78(1):1--12, 2008.

\bibitem{GPPS:14}
Alexandre Girouard, Leonid Parnovski, Iosif Polterovich, and David~A. Sher.
\newblock The {S}teklov spectrum of surfaces: asymptotics and invariants.
\newblock {\em Math. Proc. Cambridge Philos. Soc.}, 157(3):379--389, 2014.

\bibitem{GP:17}
Alexandre Girouard and Iosif Polterovich.
\newblock Spectral geometry of the {S}teklov problem (survey article).
\newblock {\em J. Spectr. Theory}, 7(2):321--359, 2017.

\bibitem{GuillarmouGuillope:07}
Colin Guillarmou and Laurent Guillop\'{e}.
\newblock The determinant of the {D}irichlet-to-{N}eumann map for surfaces with
  boundary.
\newblock {\em Int. Math. Res. Not. IMRN}, (22):Art. ID rnm099, 26, 2007.

\bibitem{HenkinMichel:07}
Gennadi Henkin and Vincent Michel.
\newblock On the explicit reconstruction of a {R}iemann surface from its
  {D}irichlet-{N}eumann operator.
\newblock {\em Geom. Funct. Anal.}, 17(1):116--155, 2007.

\bibitem{Jollivet:14}
Alexandre Jollivet and Vladimir Sharafutdinov.
\newblock On an inverse problem for the {S}teklov spectrum of a {R}iemannian
  surface.
\newblock In {\em Inverse problems and applications}, volume 615 of {\em
  Contemp. Math.}, pages 165--191. Amer. Math. Soc., Providence, RI, 2014.

\bibitem{Jollivet:18}
Alexandre Jollivet and Vladimir Sharafutdinov.
\newblock Steklov zeta-invariants and a compactness theorem for isospectral
  families of planar domains.
\newblock {\em J. Funct. Anal.}, 275(7):1712--1755, 2018.

\bibitem{Keen:74}
Linda Keen.
\newblock Collars on {R}iemann surfaces.
\newblock In {\em Discontinuous groups and {R}iemann surfaces ({P}roc. {C}onf.,
  {U}niv. {M}aryland, {C}ollege {P}ark, {M}d., 1973)}, pages 263--268. Ann. of
  Math. Studies, No. 79, 1974.

\bibitem{Khuri:91}
Hala~Halim Khuri.
\newblock Heights on the moduli space of {R}iemann surfaces with circle
  boundaries.
\newblock {\em Duke Math. J.}, 64(3):555--570, 1991.

\bibitem{Kim:08}
Young-Heon Kim.
\newblock Surfaces with boundary: their uniformizations, determinants of
  {L}aplacians, and isospectrality.
\newblock {\em Duke Math. J.}, 144(1):73--107, 2008.

\bibitem{Kirchhoff}
G.~Kirchhoff.
\newblock {\"U}ber die {A}ufl\"osung der {G}leichungen, auf welche man bei der
  untersuchung der linearen verteilung galvanischer {S}tr\"ome gef\"uhrt wird.
\newblock {\em Ann. Phys. Chem.}, 72:497--508, 1847.

\bibitem{KontsevichVishik:93}
Maxim Kontsevich and Simeon Vishik.
\newblock Geometry of determinants of elliptic operators.
\newblock In {\em Functional analysis on the eve of the 21st century, {V}ol.\ 1
  ({N}ew {B}runswick, {NJ}, 1993)}, volume 131 of {\em Progr. Math.}, pages
  173--197. Birkh\"auser, 1993.

\bibitem{LassasUhlmann:01}
Matti Lassas and Gunther Uhlmann.
\newblock On determining a {R}iemannian manifold from the
  {D}irichlet-to-{N}eumann map.
\newblock {\em Ann. Sci. \'{E}cole Norm. Sup. (4)}, 34(5):771--787, 2001.

\bibitem{OPS:89}
B.~Osgood, R.~Phillips, and P.~Sarnak.
\newblock Moduli space, heights and isospectral sets of plane domains.
\newblock {\em Ann. of Math. (2)}, 129(2):293--362, 1989.

\bibitem{RaySinger:73}
D.~B. Ray and I.~M. Singer.
\newblock Analytic torsion for complex manifolds.
\newblock {\em Ann. of Math. (2)}, 98:154--177, 1973.

\bibitem{SWY:80}
R.~Schoen, S.~Wolpert, and S.~T. Yau.
\newblock Geometric bounds on the low eigenvalues of a compact surface.
\newblock In {\em Geometry of the {L}aplace operator ({P}roc. {S}ympos. {P}ure
  {M}ath., {U}niv. {H}awaii, {H}onolulu, {H}awaii, 1979)}, Proc. Sympos. Pure
  Math., XXXVI, pages 279--285. Amer. Math. Soc., Providence, R.I., 1980.

\bibitem{Seeley:69}
R.~Seeley.
\newblock The resolvent of an elliptic boundary problem.
\newblock {\em Amer. J. Math.}, 91:889--920, 1969.

\bibitem{Seeley:67}
R.~T. Seeley.
\newblock Complex powers of an elliptic operator.
\newblock In {\em Singular {I}ntegrals ({P}roc. {S}ympos. {P}ure {M}ath.,
  {C}hicago, {I}ll., 1966)}, pages 288--307. Amer. Math. Soc., Providence,
  R.I., 1967.

\bibitem{Tutte:01}
W.~T. Tutte.
\newblock {\em Graph theory}, volume~21 of {\em Encyclopedia of Mathematics and
  its Applications}.
\newblock Cambridge University Press, Cambridge, 2001.
\newblock With a foreword by Crispin St. J. A. Nash-Williams, Reprint of the
  1984 original.

\bibitem{Weisberger:87}
William~I. Weisberger.
\newblock Conformal invariants for determinants of {L}aplacians on {R}iemann
  surfaces.
\newblock {\em Comm. Math. Phys.}, 112(4):633--638, 1987.

\bibitem{Wentworth:91b}
Richard~A. Wentworth.
\newblock Asymptotics of determinants from functional integration.
\newblock {\em J. Math. Phys.}, 32(7):1767--1773, 1991.

\bibitem{Wentworth:08}
Richard~A. Wentworth.
\newblock Precise constants in bosonization formulas on {R}iemann surfaces.
  {I}.
\newblock {\em Comm. Math. Phys.}, 282(2):339--355, 2008.

\bibitem{Wentworth:12}
Richard~A. Wentworth.
\newblock Gluing formulas for determinants of {D}olbeault {L}aplacians on
  {R}iemann surfaces.
\newblock {\em Comm. Anal. Geom.}, 20(3):455--499, 2012.

\bibitem{Wolpert:87}
Scott~A. Wolpert.
\newblock Asymptotics of the spectrum and the {S}elberg zeta function on the
  space of {R}iemann surfaces.
\newblock {\em Comm. Math. Phys.}, 112(2):283--315, 1987.

\bibitem{Wolpert:90}
Scott~A. Wolpert.
\newblock The hyperbolic metric and the geometry of the universal curve.
\newblock {\em J. Differential Geom.}, 31(2):417--472, 1990.

\end{thebibliography}

 \end{document}